\documentclass[12pt]{amsart}
\usepackage[top=1in, bottom=1in, left=1in, right=1in]{geometry}
\usepackage{dsfont}
\usepackage{amsmath}
\usepackage{amsthm}
\usepackage{amssymb}
\usepackage{multirow}
\usepackage{mathtools}
\usepackage{times}
\usepackage[utf8]{inputenc}
\usepackage[english]{babel}

\usepackage{indentfirst}  
\usepackage{graphicx}
\usepackage{color}
\usepackage[colorlinks=true]{hyperref}
\hypersetup{colorlinks=true, citecolor=green, linkcolor=green, filecolor=magenta, urlcolor=cyan}
\usepackage[shortlabels]{enumitem}
\usepackage{titlesec}
\titleformat{\chapter}[display]{\bfseries\huge}{\filright\huge\chaptertitlename~\thechapter}{3ex}{\titlerule\vspace{1ex}\filright}[\vspace{1ex}\titlerule]

\newtheorem{theorem}{Theorem}[section]
\newtheorem{prop}[theorem]{Proposition}
\newtheorem{lemma}[theorem]{Lemma}
\newtheorem{corollary}[theorem]{Corollary}
\newtheorem{conjecture}[theorem]{Conjecture}
\theoremstyle{definition}
\newtheorem{definition}{Definition}[section]
\theoremstyle{definition}

\theoremstyle{remark}
\newtheorem{remark}[theorem]{Remark}

\newenvironment{feqn*}{\begin{mdframed}\begin{equation*}}{\vspace{1mm}
\end{equation*}\end{mdframed}}

\numberwithin{equation}{section}

\newcommand{\N}{\mathbb{N}}
\newcommand{\Z}{\mathbb{Z}}

\newcommand{\F}{\mathbb{F}}
\newcommand{\CA}{\mathcal{A}}

\newcommand{\CB}{\mathcal{B}}

\newcommand{\CD}{\mathcal{D}}

\newcommand{\CH}{\mathcal{H}}

\newcommand{\CM}{\mathcal{M}}
\newcommand{\CN}{\mathcal{N}}

\newcommand{\CP}{\mathcal{P}}

\newcommand{\CS}{\mathcal{S}}

\DeclareMathOperator{\prob}{{\bf Prob}}

\DeclareMathOperator{\lcm}{lcm}

\newcommand{\bs}\boldsymbol{}
\renewcommand{\geq}{\geqslant}
\renewcommand{\leq}{\leqslant}

\renewcommand{\tilde}{\widetilde}

\renewcommand{\mod}[1]{\,({\rm mod}\,#1)}
\newcommand{\vect}[1]{\overrightarrow{\boldsymbol{#1}}}

\definecolor{blue}{rgb}{.2,.6,.75}
\definecolor{green}{rgb}{.4,.7,.4}
\definecolor{red}{rgb}{1,0,0}

\titleformat{\section}[block]{\scshape\centering}{\arabic{section}.}{1ex}{}{}

\begin{document}




\title[On The Moments of the Number of Representations as Sums of Two Prime Squares]{On The Moments of the Number of Representations as Sums of Two Prime Squares}

\author{Cihan Sabuncu}

\address{D\'epartement de math\'ematiques et de statistique\\
Universit\'e de Montr\'eal\\
CP 6128 succ. Centre-Ville\\
Montr\'eal, QC H3C 3J7\\
Canada}


\email{cihan.sabuncu@umontreal.ca}

\subjclass[2010]{}

\date{\today}

\begin{abstract}
We study the moments of the function that counts the number of representations of an integer as sums of two prime squares. We refine some of the previous arguments and apply the Selberg sieve to get an unconditional upper bound for all moments. We also prove a lower bound for all moments conditional on some generalization of the Green-Tao theorem on linear equations in primes. More precisely, for the fifth moment and onward, we get the expected order of magnitude lower and upper bounds. In addition, we provide some heuristics on the mass function of this representation function.
\end{abstract}

\maketitle

\section{Introduction}

Let $\mathcal{P}$ denote the set of all primes. For $n\in \N $, we define the representation as sum of two squares function in the first quadrant
$$ r_0(n) :=\#\{ (a,b)\in \Z^2 : n=a^2+b^2 , a\geq 0 , b>0 \} .$$
Then, we know that
\begin{equation}\label{number_of_representations}
\#\{(u,v)\in \Z^2 : u^2+v^2 = n\} = 4 \cdot r_0(n) ,
\end{equation}
and
$$ r_0(n)= \sum_{d|n} \chi_4(d), $$
where $\chi_4$ is the non-principal Dirichlet character $\bmod$ 4. On restricting both of the squares to be squares of primes we set
$$r_2(n) := \# \{ (p,q)\in \mathcal{P}^2 : p^2+q^2=n \} . $$
We are interested in understanding how these functions behave. For instance, by classical complex analytic techniques, one can show, for any $k\in \mathbb{N}$,
\begin{equation}
\label{Perron_Moments}
\sum_{n\leq x} r_0^k (n) = a_k x (\log x)^{2^{k-1}-1} + O_k \big( x (\log x)^{2^{k-1}-2} \big),
\end{equation}
for some explicit $a_k >0$. A more general result due to Blomer-Granville \cite{MR2267284} gives a uniform formula for the representation function of any binary quadratic form.

Hence, we know how $r_0(n)$ is distributed on $\N$. We are interested in the analogous effect for $r_2(n)$. However, in the case of $r_2$, we are not dealing with a multiplicative function and classical tools can not be applied directly. We study the moments of this function and get upper and lower bounds. In addition, we give a detailed heuristic for the mass function of a related representation function.

\noindent The first moment of $r_2$ can be computed by using the Prime Number Theorem:
\begin{equation} \label{first_moment}
\sum_{n\leq x} r_2(n) = \frac{\pi x}{\log^2 x} + O \bigg( x \frac{\log\log x}{\log^3 x} \bigg) ;
\end{equation}
see Plaksin \cite[Lemma 11]{MR616225} and Li \cite[Lemma 3.3]{MR4054578}.

\noindent For higher moments, previous results relied on obtaining upper bounds for the cardinality of the set
$$ \CD_k (x) := \{ (\vect{p} , \vect{q})\in (\CP\cap [0,\sqrt{x}])^{2k}: p_i^2+q_i^2=p_j^2+q_j^2 , \{ p_i , q_i \}\not= \{p_j , q_j\}, \forall i\not=j \}, $$
which we call the \emph{non-diagonal solutions of level $k$}. Thence, the $k$-th moment of $r_2$ can be written as a linear combination of the number of non-diagonal solutions of every level $\leq k$.

\noindent For instance, the first result on the non-diagonal solutions was proved by Erd\H{o}s \cite{Erdos1938}, who obtained that:
$$ \CD_2(x) \ll x \frac{(\log\log x)^5}{\log^3 x} , $$
thus obtaining an asymptotic formula for the second moment (using $\eqref{first_moment}$)
$$
\sum_{n\leq x} r_2^2(n) =  \frac{2\pi x}{\log^2 x} + O \bigg( x \frac{(\log\log x)^5}{\log^3 x} \bigg) .
$$
Rieger \cite{MR229603} improved on Erd\H{o}s' result, obtaining a better upper bound for the set of non-diagonal solutions of level 2,
$$ \CD_2(x) \ll \frac{x}{\log^3 x}. $$
We show that this is the correct order of magnitude for this set.

\noindent More recently, Blomer-Brüdern \cite{MR2418801} got bounds on $\CD_3(x)$,
$$ \CD_3(x) \ll x \frac{(\log \log x)^6}{\log^3 x} , $$
and obtained the third moment
\begin{equation} \label{third_moment}
\sum_{n\leq x} r_2^3(n) = \frac{4\pi x}{\log^2 x} + O \bigg( x \frac{(\log\log x)^6}{\log^3 x} \bigg) . 
\end{equation}
They also state their methods are susceptible to an improvement following the argument of Rieger \cite{MR229603}. We give an improvement on the power of $\log\log x$ with our method.

\begin{remark}
The result of Blomer-Brüdern \cite{MR2418801} relies on the arithmetic of $\Z[i]$ by turning the non-diagonal solutions of level $k$ into studying multiple-variable equations in $\Z[i]$. In our case, we study the non-diagonal solutions of level $k$ by turning the initial problem to a system of linear forms and a quadratic form being prime in $\Z$. This allows us to use the Selberg sieve bounds and makes it easier to get bounds for all moments. To reduce the appearance of double-logarithmic factors we split our sum into small, medium, and large prime ranges. For small primes, we show that the contribution is negligible, for large primes, we use the Selberg sieve and a result of Henriot \cite{MR2911138} to get the upper bound. For medium primes, we do an $e$-adic decomposition and study each interval in the same way we study the large prime range.

\noindent We also note that the previous work capitalized on the fact that the diagonal solutions dominate the other solutions, this effect is called ``paucity", and that's why in these cases we have asymptotic formulas. However, the heuristic argument we give in Section $\ref{heuristic_section}$ shows that for $k\geq 4$, we expect the non-diagonal solutions of level $k$ to be the main term. Asymptotic formulas for general moments cannot be obtained by our method.
\end{remark}

\noindent To study the moments of $r_2$, we define
$$ R_2(n) := \#\{(p,q)\in \CP^2 : p<q , p^2+q^2=n\}. $$ 
Then we have 
\begin{equation}\label{splitting}
r_2(n)=2R_2(n)+1_{n\in 2\cdot \CP^2}.
\end{equation}
Notice that
\begin{align}
\sum_{n\leq x}r_2^k(n) &= 2^k\sum_{n\leq x} R_2^k(n) +O_k \bigg( \sum_{p\leq \sqrt{x/2}} R_2^{k-1}(2p^2) \bigg) \nonumber \\
&= 2^k \sum_{n\leq x} R_2^k(n) + O_k \bigg( \frac{\sqrt{x}}{\log x} \bigg), \label{r_2_to_R_2}
\end{align}
since $R_2(2p^2)\leq 4$, and $\sum_{n\leq x}1_{n \in 2\cdot \CP^2}\ll \sqrt{x}/\log x$.

\noindent We begin our study of the moments of $r_2^k$ by
\begin{equation}
\label{diagonal}\sum_{n\leq x} R_2^k(n) = \sum_{n\leq x} \sum_{\ell=1}^{k}\genfrac{\{}{\}}{0pt}{}{k}{\ell} R_2(n)(R_2(n)-1)\cdots (R_2(n)-\ell+1) = \sum_{\ell=1}^k  \genfrac{\{}{\}}{0pt}{}{k}{\ell} \CS_{\ell},
\end{equation}
where $\genfrac{\{}{\}}{0pt}{}{k}{\ell}$ are the Stirling numbers of the second kind and
$$\CS_{\ell} := \ell ! \sum_{n\leq x} \binom{R_2(n)}{\ell}.$$
Notice here that this sum is only supported on integers with $R_2(n)\geq \ell$.

\noindent Hence, studying $R_2^k$ relies on studying $\mathcal{S}_{\ell}$ for all $\ell\leq k$. We will show later that the main contribution for the $k$-th moment of $r_2$ comes from the term with $\ell=k$ for $k\geq 4$. Hence, on using $\eqref{r_2_to_R_2}$, we see that the bounds for $\CS_k$ would produce the corresponding bounds for $\CD_k(x)$.

\noindent We can write $\CS_k$ as a sum over representations in the following way
\begin{equation}\label{the_single_sum}
\CS_k = \sum_{\substack{n\leq x}} \prod_{i=0}^{k-1}(R_2(n)-i) = \sum_{(\vect{p} , \vect{q})\in \CH_k(x)} 1 .
\end{equation}
where
\begin{equation}\label{S_k_definition}
\CH_k(x) = \{ (\vect{p} , \vect{q})\in \CP^{2k} : p_i^2+q_i^2=p_j^2+q_j^2 \leq x , (p_i,q_i)\not=(p_j,q_j), \forall i\not=j ,  p_i < q_i , \forall i\leq k \} . 
\end{equation}

\noindent In this paper, we provide unconditional upper bound and conditional lower bound for $\mathcal{S}_k$ for any $k\geq 2$, and hence the corresponding bounds for $r_2^k$. The lower bounds are conditional subject to a quadratic extension of the Green-Tao \cite{MR2680398} theorem on linear equations in primes.
\begin{theorem} \label{main_result}
Let $k\in \mathbb{Z}_{\geq 2}$ and $x$ be a large number. Then,
\begin{align*}
\frac{x}{\log^3 x} \ll\, &\CS_2 \ll x \frac{(\log \log x)^2}{\log^3 x}, \\
\frac{x}{\log^3 x} \ll\, &\CS_3 \ll x \frac{(\log\log x)^2}{\log^3 x}, \\
\frac{x}{\log x} \ll\, &\CS_4 \ll x\frac{\log\log\log x}{\log x},
\end{align*}
and
$$\CS_k \asymp_k x(\log x)^{2^{k-1}-2k-1} \quad \text{ for } k\geq 5, $$
where the lower bounds are conditional on Conjecture $\ref{conjectural_lower_bound}$.
\end{theorem}
\noindent In Section $\ref{heuristic_section}$, we give a heuristic reasoning for the size of $\CS_k$, and it tells us that these double and triple logarithmic factors should not be there. Theorem $\ref{main_result}$ implies immediately by $\eqref{r_2_to_R_2}$ and $\eqref{diagonal}$ the following result.
\begin{corollary} \label{result_for_r_2}
Let $k\in \Z_{\geq 3}$ and $x\to\infty$. Then,
\begin{gather*}
\sum_{n\leq x} r_2^3(n) = \frac{4\pi x}{\log^2 x} + O\bigg( x \frac{(\log\log x)^2}{\log^3 x} \bigg), \\
\frac{x}{\log x} \ll \sum_{n\leq x} r_2^4(n) \ll x \frac{\log\log\log x}{\log x}, \\
\sum_{n\leq x} r_2^k(n) \asymp_k x(\log x)^{2^{k-1}-2k-1} \quad \text{ for } k\geq 5,
\end{gather*}
where the lower bounds are conditional on Conjecture $\ref{conjectural_lower_bound}$.
\end{corollary}
This improves on the result of Blomer-Brüdern \cite{MR2418801} for the third moment $\eqref{third_moment}$, and the bounds for $k\geq 4$ are, to the author's knowledge, the first of their kind.

\noindent The exponent grows very fast with $k$, see the following table.
\begin{table}[h!]
\caption{Growth of the exponent}
\label{growth_of_exponent}
\centering
\begin{tabular}{|c|c|c|c|c|c|c|c|c|c|c|c|c|c|c|}
\hline
$k$ & 1 & 2 & 3 & 4 & 5 & 6 & 7 & 8 & 9 & 10 & 11 & 12 & 13 & 14 \\
\hline
$2^{k-1}-2k-1$ & -2 & -3 & -3 & -1 & 5 & 19 & 49 & 111 & 237 & 491 & 1001 & 2023 & 4069 & 8163 \\
\hline
\end{tabular}
\end{table}

\begin{remark}
Acting analogously, we can obtain a result for the related representation function
$$ \sum_{n\leq x} r_1^k(n) \asymp_k x(\log x)^{2^{k-1}-k-1} \quad \text{ for } k\geq 3, $$
where $r_1(n)=\#\{(m,p)\in \N\times \CP : n=m^2+p^2\}$ and the lower bound is conditional under a suitably adjusted version of Conjecture $\ref{conjectural_lower_bound}$. The proof follows the same lines, but now we are counting $k$ linear forms and a quadratic form in primes instead of $2k$ linear forms and a quadratic form. So, we get $(\log x)^{-k-1}$ when we are counting primes. We can also bound the singular series similar to way we bound in our case. Moreover, we can also get the mixed moments
$$ \sum_{n\leq x} r_1^k(n) r_2^{\ell}(n) \asymp_{k,\ell} x (\log x)^{2^{k+\ell-1}-k-2\ell - 1} \quad \text{ for } k\geq 3 \text{ and } \ell\geq 5, $$
where the lower bound is conditional under a suitably adjusted version of Conjecture $\ref{conjectural_lower_bound}$. This bound also holds for $k< 3$ or $\ell<5$ up to some double-logarithmic factors.
\end{remark}
\noindent Now we can state our conjecture. To do this, we define first an ``admissible" representation tuple in the appropriate sense.
\begin{definition}\label{admissibility}
Let $k\in\Z_{\geq 2}$, and $\{ m_1 , n_1, \cdots , m_k , n_k \} \subset \Z$ with $m_i^2+n_i^2=m_j^2+n_j^2$, $\forall i,j\leq k$ and $(m_i, n_i)\not=(m_j , n_j)$ for $i\not=j$. Then we call $(\vect{m},\vect{n})$ an \emph{admissible representation} if
$$ \forall p\in \CP , \exists a_p \in \F_p : (a_p^2+1)\prod_{i=1}^k (m_i a_p - n_i)(n_i a_p + m_i) \not\equiv 0 \mod p . $$
\end{definition}
\noindent The idea is to use the above definition to make
\begin{equation}\label{what_we_want}
r^2+s^2,m_i r - n_i s, n_i r + m_i s\in \CP 
\end{equation}
for $i\leq k$ and some $r,s\in \N$. Now we are ready to formulate our main conjecture.
\begin{conjecture}\label{conjectural_lower_bound}
Let $k\in\Z_{\geq 2}$, and $x\to\infty$. Let $\{ m_1 , n_1, \cdots , m_k , n_k \} \subset \Z$ with $(m_i, n_i)\not=(m_j , n_j)$ for $i\not=j$. If $(\vect{m},\vect{n})$ is an admissible representation, then
\begin{align}
\#\{ (r,s)\in\N^2 : r^2+s^2\leq x ,\,& 0 < m_i r - n_i s < n_i r + m_i s  , \eqref{what_we_want} \text{ holds, }\forall i\leq k \}  \label{sieving_set} \\
&\gg_k \mathfrak{S}(m_1, \cdots , n_k)\frac{x}{(\log x)^{2k+1}} \nonumber
\end{align}
where
$$ \mathfrak{S}(m_1, \cdots , n_k) = \prod_p \bigg( 1 - \frac{\nu_p(m_1,\cdots ,n_k)}{p^2} \bigg) \bigg( 1 - \frac{1}{p} \bigg)^{-2k-1} $$
and
$$ \nu_p(m_1,\cdots ,n_k) = \#\{ (r,s)\in \F_p^2: (r^2+s^2)\prod_{i=1}^k (m_i r - n_i s)( n_i r + m_is)\equiv 0 \mod p \}. $$
\end{conjecture}
\noindent It should be noted that the definition of $\nu_p(m_1,\cdots , n_k)$ is closely related to Definition $\ref{admissibility}$, and we can get from the former to the latter by looking at the projective solutions.
\begin{remark}
In our proof we use the conjectural lower bound of the correct order of magnitude, however, any lower bound on the set in $\eqref{sieving_set}$ yields a lower bound on $\CS_k$, and thus on the moments of $r_2$.
\end{remark}
We define the mass functions $\CN_r(x)=\#\{n\leq x : R_2(n)=r \}$ of $R_2$. Our detailed heuristic argument gives us the following,
\begin{conjecture} \label{CN_r}
We have
$$ \CN_2(x) \sim \rho \frac{x}{\log^3 x} $$
where $\rho \approx 0.0282\cdots$(defined in Section \ref{heuristic_section}), and for $3\leq r =o((\log x)^2)$,
\begin{equation}
\CN_r(x) \sim \phi_r \bigg( \frac{2\log\log x}{\log 2} \bigg) \frac{x}{(\log x)^{3+2\delta}\sqrt{\log\log x}},
\end{equation}
where, for $\beta=2^{1-t}$,
$$ \phi_r(t) := \frac{\tilde\kappa}{r!}\sum_{k\in \Z} (2^k \beta)^{r-2-\tau} e^{-2^k \beta} $$
is a 1-periodic, bounded function. $\tilde\kappa \approx 0.02761\cdots$(defined in Section \ref{heuristic_section}.), $\tau=\frac{\log(\frac{1}{\log 2})}{\log 2}=0.5287 \cdots$, and $\delta=1- \frac{1+\log\log 2}{\log 2}=0.086071\cdots$ is the Erd\H{o}s-Tenenbaum-Ford constant.
\end{conjecture}
For $r$ sufficiently large, we also expect to find infinitely many $x^+, x^-$ such that $\CN_r(x^+)>\CN_{r+1}(x^+)$ and $\CN_r(x^-)<\CN_{r+1}(x^-)$.

In Section $\ref{heuristic_section}$, we will give our heuristic reasoning for this conjecture and bounds on $\CS_k$. We will give our preliminary lemmas in Section $\ref{third_section}$. We will study the sum $\eqref{the_single_sum}$ and give a Selberg sieve bound with our singular series in Section $\ref{fourth_section}$. We will prove the upper bound of Theorem $\ref{main_result}$ successively in Sections $\ref{sixth_section}$, $\ref{seventh_section}$ and $\ref{eight_section}$ and the lower bound in Section $\ref{ninth_section}$.

\noindent \textbf{Acknowledgements.} The author is grateful to Andrew Granville for his continued guidance and suggestions. He would also like to thank Alisa Sedunova for her careful read through, and Tony Haddad, Jonah Klein, Sun-Kai Leung, Siva Sankar Nair, Stelios Sachpazis, Jeremy Schlitt, and Christian Táfula for their advices and discussions.

\section{Heuristics for $\CS_k$ and the Anatomy of $R_2$} \label{heuristic_section}
We know that $R_2(n)>0$ only for those $n$ who do not have $p^k\|n$ such that $p\equiv 3\mod 4$ and $k$ odd. Moreover, if $k$ is even, then it can only be $2$ and it happens when $n=2p^2$. Thus, for simplicity, we shall show our heuristics for squarefree $n$, but it applies more generally. Since $n$ is sum of two prime squares, it is almost surely divisible by $2$, hence we define $\omega^*(n)$ to be the number of odd prime factors of $n$. \\
\noindent Fix some $m\in\mathbb{N}$. Then we know
\begin{equation}\label{omega_m_squares} 
\{ n\in \mathbb{N} : \omega^*(n)=m \} \asymp_k \frac{x}{\log x} \frac{(\log\log x)^{m-1}}{(m-1)!},
\end{equation}
when $m\leq C_k \log\log x$. For a reference, see Koukoulopoulos \cite[Theorem 16.2]{MR3971232}. When $m>C_k \log\log x$, we use the Hardy-Ramanujan theorem
$$\{n\in \mathbb{N} : \omega^*(n)=m \} \ll \frac{x}{\log x} \frac{(\log\log x + D)^{m-1}}{(m-1)!}, $$
for some constant $D$. 

\noindent Now, we want $n=\square + \square$, thus its prime factors must be $2$ or $1\mod 4$. Now, an odd prime is either $1\mod 4$ or $3\mod 4$. Hence, we expect a prime factor to be $1\mod 4$ half the time. Since $\omega^*(n)=m$, we get the same factor for each prime divisor, so we have $ \frac{1}{2^m}$. Therefore, we expect
\begin{align*}
&\#\{ n\in\mathbb{N} : n\in \CM , \omega^*(n)=m \} \asymp_k \frac{x}{\log x} \frac{\big(\frac{1}{2} \log\log x \big)^{m-1}}{(m-1)!} &&\text{for } m\leq C_k \log\log x , 
\end{align*}
and
\begin{equation}\label{M_and_constant}
\CM:=\{n\in \N : 2\|n , p|\frac{n}{2}\implies p\equiv 1\mod 4\}.
\end{equation}
While for $m>C_k \log\log x$ we have
$$\#\{ n\in\mathbb{N} : n\in \CM , \omega^*(n)=m \} \ll \frac{x}{\log x} \frac{\big(\frac{1}{2} \log\log x + D \big)^{m-1}}{(m-1)!} .$$
For a reference in the literature on bounds like these, see Norton \cite{MR419382, MR523395}. Since $\omega^*(n)=m$, then we know such an $n$ has $2^{m-1}$ many representations as sums of two squares $a^2+b^2$ with $a<b$. From the prime number theorem, we know the density of primes up to $\sqrt{x}$ are $\frac{1}{\log \sqrt{x}}$, hence we expect a random representation $n=\square + \square$ to be both prime with probability $\sim \frac{1}{\log^2 \sqrt{x}}$. Now we want to choose $k$ distinct representations that are sums of two prime squares. Thus we expect for such $n$
$$  \binom{2^{m-1}}{k} \bigg( \frac{1}{\log^2 \sqrt{x}} \bigg)^k \asymp_k \frac{2^{(m-1)k}}{\log^{2k} x} $$ 
many $k$ distinct prime representations on average. So, we have to choose $C_k$ so that the numbers with $\omega^*(n)>C_k \log\log x$ are part of the tail event and are negligible. Consequently, we have for integers with $\omega^*(n)> C_k \log\log x$ and $C_k = 3 \cdot 2^{k-1}$,
\begin{align*}
\sum_{m> C_k \log\log x} \frac{x}{\log x} \frac{(\frac{1}{2}\log\log x + D)^{m-1}}{(m-1)!} \binom{2^{m-1}}{k} \frac{1}{\log^{2k} x} &\ll_k \frac{x}{(\log x)^{2k+1}}\sum_{m> C_k \log\log x} \frac{(2^{k-1}\log\log x )^{m-1}}{(m-1)!} \\
&\ll_k  \frac{x}{(\log x)^{2k+1}} \frac{(2^{k-1}\log\log x)^{C_k \log\log x}}{(C_k \log\log x)!} \\
&\leq \frac{x}{(\log x)^{C_k(\log 3 - 1) + 2k+1}}
\end{align*}
where the sum on the first line is dominated by the first term because 
$$ \frac{(2^{k-1}\log\log x)^m}{m!} \leq \frac{2^{k-1}}{C_k} \cdot \frac{(2^{k-1}\log\log x)^{m-1}}{(m-1)!} = \frac{1}{3} \cdot \frac{(2^{k-1}\log\log x)^{m-1}}{(m-1)!} , $$
since $m\geq C_k \log\log x$ and $C_k = 3\cdot 2^{k-1}>2^{k-1}$. Thus, by Stirling's approximation, we get the bound. \\
\noindent Therefore, putting the sums over $m\leq C_k \log\log x$ and $m>C_k \log\log x$, we get
\begin{align*}
\CS_k &\asymp_k \sum_{\frac{\log k}{\log 2} \leq m \leq C_k \log\log x} \frac{x}{\log x} \frac{(\frac{1}{2}\log\log x)^{m-1}}{(m-1)!} \binom{2^k}{m} \frac{1}{\log^{2k} x} + O_k \bigg( \frac{x}{(\log x)^{C_k(\log 3 - 1) + 2k+1}} \bigg) \\
&\asymp_k \frac{x}{(\log x)^{2k+1}} \sum_{\frac{\log k}{\log 2} \leq m \leq C_k \log\log x} \frac{(2^{k-1}\log\log x)^{m-1}}{(m-1)!} \\
&\asymp_k x (\log x)^{2^{k-1}-2k-1},
\end{align*}
where in the last line we completed the sum to $m\leq C_k \log\log x$ since the contribution of small $m$ satisfies
$$ \sum_{m < \frac{\log k}{\log 2}} \frac{(2^{k-1} \log\log x)^{m-1}}{(m-1)!} \ll_k (\log\log x)^{\frac{\log k}{\log 2}-1}.$$

We note that for most $n\in \N$ with $R_2(n)>0$, we have $R_2(n)=1$. To see this, let 
$$\CN_r(x)=\#\{n\leq x : R_2(n) = r\}, $$
then we have
$$ \sum_{n\leq x} R_2^k(n)= \sum_{r=1}^{\infty} r^k \CN_r(x).  $$
Then by Cauchy-Schwarz we have
$$ \bigg(\sum_{n\leq x}R_2(n)\bigg)^2\bigg/ \bigg( \sum_{n\leq x} R_2^2(n)\bigg) \leq \CN_1(x) \leq \sum_{n\leq x} R_2(n) ,$$
which gives
\begin{equation} \label{CN_1}
\CN_1(x)\sim \frac{\pi}{2} \frac{x}{\log^2 x}.
\end{equation}
But we have
\begin{equation}\label{CN_2}
\sum_{r=2}^{\infty} \CN_r(x) = \sum_{\substack{n\leq x \\ R_2(n)\geq 2}} 1 \leq \CS_2 \ll \frac{x}{\log^3 x} 
\end{equation}
by the result of Rieger \cite{MR229603}. We give a heuristic that shows $\CN_2(x)$ is of this size. However, the rest $\CN_r(x)$ for $r\geq 3$ are much smaller with some similarity to the multiplication table problem.

\noindent For an integer $n\leq x$ the typical representation $n=a^2+b^2$ have 1 in $\log^2 \sqrt{x}$ chance of having both $a$ and $b$ prime. Moreover, an integer $n$ with $\omega^*(n)=m$ has $2^{m-1}$ many representations $n=a^2+b^2$ with $a<b$ and $\gcd(a,b)=1$. Therefore, we expect for a typical integer $n\leq x$,
$$ \prob ( R_2(n) = r) = \binom{2^{m-1}}{r} \bigg( \frac{1}{\log^2 \sqrt{x}}\bigg)^{r} \bigg(1 - \frac{1}{\log^2 \sqrt{x}}\bigg)^{2^{m-1}-r} . $$
Now, for $r = o((\log x)^2)$, we have $(1- \frac{1}{\log^2 \sqrt{x}})^r = 1 + o(1)$. Hence we expect
$$ \prob(R_2(n)=r) \sim \frac{2^{(m+1)r}}{r!} \frac{1}{\log^{2r} x} \bigg(1-\frac{4}{\log^2 x}\bigg)^{2^{m-1}} \sim \frac{2^{(m+1)r}}{r!} \frac{1}{\log^{2r} x} \exp\bigg( - \frac{2^{m+1}}{\log^2 x}\bigg). $$
As in \eqref{omega_m_squares}, we have
\begin{equation}\label{inequality_to_asymptotic}
\#\{ n\leq x : n\in \CM , \omega^*(n)=m \} \asymp \frac{x}{\log x} \frac{(\frac{1}{2}\log\log x)^{m-1}}{(m-1)!}.
\end{equation}
So, we expect
\begin{align*}
\#\{ n\leq x : \omega^*&(n) = m ,\, R_2(n)=r \} \asymp \frac{x}{\log x} \frac{(\frac{1}{2}\log\log x)^{m-1}}{(m-1)!} \cdot \frac{2^{(m+1)r}}{r!} \frac{1}{\log^{2r} x} \exp\bigg( - \frac{2^{m+1}}{\log^2 x}\bigg) \\
&= \frac{2^{2r}}{r!} \frac{x}{\log^{2r+1} x} \cdot T_m \text{, where } T_m=T_m(x) :=  \frac{(2^{r-1}\log\log x)^{m-1}}{(m-1)!} \exp \bigg(- \frac{2^{m+1}}{\log^2 x}\bigg).
\end{align*}
Now, to get $\CN_r(x)$ we want to sum over $m$. We see that the terms $\frac{(2^{r-1}\log\log x)^{m-1}}{(m-1)!}$ are increasing by a ratio of $>2\log2>1$ for $2^{m+1}\leq (\log x)^2$ and $r\geq 3$. Also, for this range of $m$, $\exp(-2^{m+1}/\log^2 x)$ is $\asymp 1$. For $r=1$, we already know $\CN_1(x)\sim \pi x/(2\log^2 x)$ by $\eqref{CN_1}$. For $r=2$, we see that for $2^{m+1}\leq (\log x)^2$, there is no one $m$ for which $T_m$ is maximal, and for $2^{m+1}>(\log x)^2$, we see that $\exp(-2^{m+1}/\log^2 x)$ and hence $T_m$ decreases rapidly. Hence, we expect since $4>2/\log 2>2$,
\begin{align}
\CN_2(x) &\sim  8c_{\lambda} \frac{x}{\log^5 x} \sum_{m\leq \lambda \log\log x} \frac{(2\log\log x)^{m-1}}{(m-1)!} \exp\bigg(-\frac{2^{m+1}}{\log^2 x} \bigg) \label{appearance_of_c_lambda} \\
&= 8 c_{\lambda} \frac{x}{\log^5 x} \sum_{m\leq \frac{2}{\log 2}\log\log x} \frac{(2\log\log x)^{m-1}}{(m-1)!} + O\bigg( \frac{x}{\log^7 x} \sum_{m\leq \frac{2}{\log 2}\log\log x} \frac{(4\log\log x)^{m-1}}{(m-1)!} \bigg)\nonumber \\
&\sim \rho \frac{x}{\log^3 x}, \nonumber
\end{align}
where we used $\exp(-2^{k+1}/\log^2 x)=1+O(2^k/\log^2 x)$, Stirling's approximation, and we have
$$ c_{\lambda} := \frac{3}{2^{\lambda + 2} \Gamma(\lambda+1)} \prod_{p\equiv 1(\bmod 4)}\bigg( 1 + \frac{2\lambda}{p-1} \bigg) \bigg(1 - \frac{1}{p} \bigg)^{\lambda} \prod_{p\equiv 3(\bmod 4)} \bigg( 1 - \frac{1}{p} \bigg)^{\lambda}, $$
and $\lambda := 2/\log 2$, and $\rho:=8c_{\lambda}$. We have \eqref{appearance_of_c_lambda} because we can get an asymptotic for \eqref{inequality_to_asymptotic} with this Euler product by following the proof of Selberg \cite{MR0067143} for numbers in $\CM$ with $\sim \lambda \log\log x$ many prime factors. \\
\noindent Now, we study $r\geq 3$. If $2^{m+1} \leq (\log x)^2$, then we have $T_m$ is increasing and for $2^{m+1}>(\log x)^2$ it starts decreasing due to the $\exp(-2^{m+1}/\log^2 x)$. Therefore, it maximizes at $m=K+ k$ for $K=\lfloor \frac{2}{\log 2} \log\log x\rfloor$ and $k=O(1)$. If we let $K=\frac{2}{\log 2}\log\log x - t$, then we have
$$ \#\{ n\leq x : \omega^*(n)=K+k , R_2(n)=r \} \sim c_{\lambda} \frac{2^{2r}}{r!} \frac{x}{\log^{2r+1} x} \frac{(2^{r-1}\log\log x)^{K+k-1}}{(K+k-1)!} \exp ( - 2^{k - t+1} ) ,$$
where we get an asymptotic the same way as \eqref{appearance_of_c_lambda}. Now, we can study the double logarithmic factor. Recall $K=\lambda \log\log x - t$, then by Stirling's approximation,
\begin{align*}
\frac{(2^{r-1}\log\log x)^{K+k-1}}{(K+k-1)!} &\sim \frac{1}{\sqrt{2\pi \log\log x}} \bigg( \frac{2^{r-2} e \log 2 }{(1-\frac{t}{\lambda \log\log x})(1+ \frac{k-1}{\lambda \log\log x})} \bigg)^{K+k-1} \\
&= \frac{1}{\sqrt{2\pi \log\log x}} \bigg( \frac{2^{r-2} e \log 2 }{\exp\big(\frac{k-t-1}{\lambda \log\log x}+O\big(\frac{1}{(\log\log x)^2}\big)\big)} \bigg)^{K+k-1}   \\
&\sim \frac{2^{(k-t-1)(r-2-\tau)}}{\sqrt{2\lambda \pi}} \frac{(\log x)^{2r - 2 - 2\delta}}{\sqrt{\log\log x}},
\end{align*}
where
$$ \delta = 1 - \frac{1+\log\log 2}{\log 2} $$
is the Erd\H{o}s-Tenenbaum-Ford constant, and $\tau=\frac{\log(\frac{1}{\log 2})}{\log 2}$. Therefore, if we sum over $k$ and put everything together, we get Conjecture \ref{CN_r} for $3\leq r =o((\log x)^2)$,
$$ \CN_r(x) \sim \phi_r \bigg( \frac{2\log\log x }{\log 2}\bigg) \frac{x}{(\log x)^{3+2\delta}\sqrt{\log\log x}},  $$
where, for $\beta=2^{1-t}$,
$$ \phi_r(t) := \frac{\tilde\kappa}{r!}\sum_{k\in \Z} (2^k \beta)^{r-2-\tau} e^{-2^k \beta} $$
is a 1-periodic, bounded function, and $\tilde\kappa = c_{\lambda}\sqrt{(2\lambda)^3/\pi}\approx 0.02761\cdots$ (Calculated by taking the first $10^6$ primes in the product.), $\tau=\frac{\log(\frac{1}{\log 2})}{\log 2}$.

We also expect, for fixed large enough $r$, both $\CN_r(x)< \CN_{r+1}(x)$ and $\CN_r(x)>\CN_{r+1}(x)$ events to happen infinitely often. We can see this is true if we have $\phi_r(t)>\phi_{r+1}(t)$ or $\phi_r(t)<\phi_{r+1}(t)$ for $t\in[0,1]$. We focus on $\phi_r(t)>\phi_{r+1}(t)$. Thus, we want
\begin{equation}\label{N_r_to_N_r+1}
f_{r-2-\tau}(\beta) > \frac{1}{r+1} f_{r-1-\tau}(\beta) \quad \text{ where } f_{R}(\beta)=\sum_{k\in \Z} (2^k\beta)^R e^{-2^k\beta},
\end{equation}
for $\beta\in (1,2]$. Recall $f_R(\beta)$ is unchanged under the action $\beta \to 2\beta$, so this is defined for all $\beta \in (0,\infty)$. Note that the sum defining $f_R(\beta)$ is maximized when the ratio of consecutive terms is $>1$, that is $2^k \beta > R \log 2$. Hence, we can take $\beta=\lambda R$ for $\lambda \in (\log 2 , \log 4]$. Using this $\beta$, we can obtain
$$ f_R(R\log 2) \sim 2^{1-\delta R} \cdot R^R e^{-R} \leq f_R(\beta) \leq  R^R e^{-R} \sim f_R(R) , $$
as $R\to\infty$. Thence, if we use this and put $R=r-1-\tau$ in \eqref{N_r_to_N_r+1}, we get for $r$ sufficiently large
$$ \lambda(r-1-\tau) \leq r+1 .$$
So, we can, for instance, take $\lambda=1$ and $\beta= (r-1-\tau)$, then for $r$ sufficiently large we get
$$ f_{r-2-\tau}(r-1-\tau) > \frac{1}{r+1}f_{r-1-\tau}(r-1-\tau). $$
Therefore, for such a sufficiently large $r$, we can choose infinitely many $x$ values for which $\CN_r(x)>\CN_{r+1}(x)$. Similarly, one can choose $\lambda>1$ and find $x$ values for large enough $r$, for which $\CN_r(x)<\CN_{r+1}(x)$. (See the subsequent work \cite{granville2023multiplication} for a similar and more detailed account of these calculations.)

\section{Preliminary Results and Notation}\label{third_section}
Throughout the paper we use the usual asymptotic notations $O(\cdot ) , \ll, \gg$ and subscript is used to indicate the dependency of the constant. We fix some $k\in \Z_{\geq 2}$ and let $x$ be a sufficiently large number.

\noindent For convenience, when we denote sums with letters, we won't be showing the dependency on $x$ and potentially $k$. For instance, we may write $M = \sum_{n\leq x} r_0^k(n)$ instead of $M(x)$. We use $(\cdot , \cdot)$ for a point in $\mathbb{R}^2$, and $\gcd(\cdot , \cdot)$ for the greatest common divisor of two numbers.

\noindent We use $p,q$ with and without subscripts for primes, $m,n,N$ with and without subscripts for integers and $r,s, \ell , c, d , D , a, b$ with and without subscripts for natural numbers, and $\vect{m}, \vect{n}$ for vectors.

\noindent We will also use $\mathbb{P} (\mathbb{F}_p)$ for the projective space over $\mathbb{F}_p$, and $[k]:= \{ 1,2,\cdots ,k \}$.

\noindent Let $\tau(n)$ be the divisor function and $\tau_{\ell}(n):=\#\{n=d_1 \cdots d_{\ell}\}$.

Now, we give our preliminary lemmas.

\noindent Firstly, we will need the following logarithmic average bound for $r_0^k$ to apply for most of our sums,
\begin{lemma}
\label{r_2_bound}
Let $k\in \mathbb{Z}_{\geq 0}$ and $x\to\infty$. Then we have
$$ \sum_{n\leq x} \frac{r_0^k(n)}{n} \ll_k (\log x)^{2^{k-1}} . $$
\end{lemma}
\begin{proof}
We can prove a stronger result by using $\eqref{Perron_Moments}$. But we provide a proof here for completeness. Since $r_0$ is multiplicative, we have
\begin{align*}
\sum_{n\leq x} \frac{r_0^k(n)}{n} &\leq \prod_{p\leq x} \bigg( 1 + \frac{r_0^k(p)}{p}+O_k\bigg(\frac{1}{p^2}\bigg) \bigg) \\
&\ll_k \prod_{\substack{p\leq x \\ p\equiv 1 \mod 4}} \bigg( 1 + \frac{2^k}{p}\bigg) \\
&\ll_k \prod_{\substack{p\leq x \\ p\equiv 1 \mod 4}} \bigg( 1 - \frac{1}{p}\bigg)^{-2^k} \ll_k (\log x)^{2^{k-1}},
\end{align*}
where we use on the last line Mertens' theorem
\begin{equation} \label{Merten_estimate}
\prod_{\substack{p\leq x \\ p\equiv 1 \mod 4}}\bigg( 1 -\frac{1}{p}\bigg)^{-1} \ll (\log x)^{1/2}. \qedhere
\end{equation}
\end{proof}
\noindent We will also need the following inequality,
\begin{lemma} \label{r_2_inequality}
Let $m , n \in \mathbb{N}$ with $m, n\in \CM $ (see \eqref{M_and_constant}), then
$$ r_0(mn) \leq r_0(m) r_0(n) . $$
\end{lemma}
\begin{proof}
Notice that $r_0(2n)=\sum_{\substack{d|2n \\ d \text{ odd}}} \chi_4(d)=r_0(n)$ for any $n\in \N$ since $\chi_4$ is supported on odd integers. Hence, we can assume without loss of generality $m,n$ are odd. So, for $m,n\in \CM$ odd, we have
$$ r_0(mn) = \sum_{d|mn} \chi_4(d) = \sum_{d|mn} 1 \leq \bigg( \sum_{d|m} 1 \bigg) \bigg( \sum_{d|n} 1 \bigg) = r_0(m) r_0(n) $$
where $\chi_4 $ is the non-principal character $\bmod 4$. 
\end{proof}
\noindent We later use this to split the $r_0$ over divisors.

\noindent Later on, we will need more refined results on averages of $r_0$. For this we will use a result of Henriot \cite[Theorem 5]{MR2911138}. There was an erratum in this article, the statement of Theorem 5 has been changed (see Henriot \cite{MR3254599}). However, it is the same statement in the case of one irreducible polynomial. We state it here.
\begin{theorem}\label{Henriot_1}
Let $Q \in \Z[X]$ be an irreducible polynomial of degree $g$. Assume that $Q$ is primitive, that is the greatest common divisor of its coefficients is 1. Let $0<\alpha <1$, $0<\delta < 1$, $A\geq 1$ and $B\geq 1$.
Let also $0<\varepsilon < \frac{\alpha}{50g(g+\frac{1}{\delta})}$ and $F$ be a function such that
$$F(a b ) \leq \min (A^{\Omega(a)}, B a^{\varepsilon})F(b)$$
for all $a , b $ with $\gcd (a, b)=1$. Define
\begin{align*}
\rho(M) &:= \#\{ n (\bmod M) : Q(n) \equiv 0 \mod M  \} .
\end{align*} 
Then we have, uniformly in $x\geq c_0 \|Q\|^{\delta}$, where $\|Q\|$ is the sum of the coefficients in absolute value of the polynomial, and $x^{\alpha} \leq y \leq x$,
$$ \sum_{x\leq n \leq x+y} F(|Q (n)|) \ll y \prod_{g<p\leq x} \bigg( 1 - \frac{\rho(p)}{p} \bigg) \sum_{n \leq x} F(n) \frac{\rho(n)}{n}  , $$
where $c_0$ and the implicit constant depend at most on $g,\alpha , \delta, A,B$.
\end{theorem}
\noindent Applied to our case, we have,
\begin{lemma} \label{Henriot_use}
Let $k\in \mathbb{Z}_{\geq 0}$ and $x,y$ be large numbers with $y\leq x$. Let $d\in \N$. We have
$$ \sum_{\substack{ a^2+b^2 \leq x \\ d|a \\ P^+(a^2+b^2)\leq y}} r_0^k ( a^2+b^2) \ll_k \frac{x}{\log x} \frac{(\log y)^{2^k}}{d^{\frac{1}{2}}}, $$
where $P^+(n)$ is the largest prime factor of $n$.
\end{lemma}
\begin{proof}
Firstly, we note that
$$ S(x) = \sum_{\substack{ a^2+b^2 \leq x \\ d|a \\ P^+(a^2+b^2)\leq y}} r_0^k ( a^2+b^2) \leq \sum_{\substack{a\leq \sqrt{x} \\ d| a}} \sum_{\substack{b\leq \sqrt{x} \\ P^+(a^2+b^2)\leq y}} r_0^k(a^2+b^2). $$
We are going to apply Theorem $\ref{Henriot_1}$, so let us fix $a$ and consider $Q(m)=m^2+a^2$ an irreducible primitive polynomial, then $b\geq c \|Q\|^{\delta}$ and we have $c \|Q\| \leq c_0 a^{2\delta}$, so we just need $b\geq c_0 a^{2\delta}$ for $\delta = \frac{1}{4}$. Take $F(n)= r_0^k(n)1_{P^+(n)\leq y}$, then this is multiplicative. Let $\gcd(u,v)=1$, then $r_0^k(u v) = r_0^k(u) r_0^k(v) \ll_{k, \varepsilon} u^{\varepsilon} r_0^k(v)$ and also $r_0(u) \leq 2^{\Omega(u)}$, so it satisfies the bounds for $F$ in Theorem $\ref{Henriot_1}$ for some $B\geq 1$ and $A=2^k$. Thus by applying it we get
\begin{align*}
\sum_{\substack{c_0 \sqrt{a} \leq b \leq \sqrt{x} \\ P^+(a^2+b^2)\leq y}}& r_0^k ( a^2+b^2) \ll_k \sqrt{x} \prod_{2<p\leq \sqrt{x}} \bigg( 1 - \frac{\rho (p)}{p} \bigg) \sum_{\substack{m\leq \sqrt{x} \\ P^+(m)\leq y}} r_0^k(m) \frac{\rho(m)}{m}.
\end{align*}
Recall the definition $\rho(n)= \#\{m\in \mathbb{Z}/n\mathbb{Z} : m^2+a^2 \equiv 0 \mod n\}$, hence we have $\rho(p) = 1+ \big( \frac{-4 a^2}{p} \big)$ for $p$ prime, where $\big(\frac{\cdot}{p}\big)$ is the Legendre symbol. Moreover it is multiplicative and $\rho(p^b)\leq 2$ for $b\geq 1$. Also for $p|a$ we have $\rho(p)=1$, and $p\nmid a$ we have $\rho(p)=1+\chi_4(p)$. Therefore, on inserting it to the product above and getting rid of higher order terms we get
\begin{align*}
\sum_{\substack{c_0 \sqrt{a} \leq b \leq \sqrt{x} \\ P^+(a^2+b^2)\leq y}}& r_0^k ( a^2+b^2)  \ll_k \sqrt{x} \prod_{p\leq \sqrt{x}} \bigg( 1 - \frac{1+\big(\frac{-4a^2}{p} \big)}{p} \bigg)\prod_{p\leq y} \bigg( 1 + \frac{r_0^k(p) ( 1 + \big( \frac{-4a^2}{p} \big))}{p} \bigg) \\
&\leq \sqrt{x} \prod_{\substack{p\leq \sqrt{x} \\ p\nmid a}} \bigg( 1 - \frac{1+\chi_4(p)}{p} \bigg)\prod_{\substack{p\leq y \\ p\nmid a}} \bigg( 1 + \frac{(1+\chi_4(p))^{k+1}}{p}  \bigg) \prod_{p|a} \bigg(1 +\frac{(1+\chi_4(p))^k}{p} \bigg) \\
&= \sqrt{x} \prod_{\substack{p \leq \sqrt{x} \\ p\equiv 1 \mod 4}} \bigg(1 - \frac{2}{p}\bigg) \prod_{\substack{p\leq y \\ p\equiv 1\mod 4}} \bigg( 1 + \frac{2^{k+1}}{p}  \bigg) \prod_{p|a} \bigg(1 - \frac{1+\chi_4(p)}{p} \bigg)^{-1} \bigg( 1 + \frac{2^k}{p}\bigg) \\
&\ll_k \sqrt{x} \prod_{\substack{p\leq \sqrt{x} \\ p\equiv 1 \mod 4}}\bigg( 1- \frac{1}{p}\bigg)^2  \prod_{\substack{p\leq y \\ p\equiv 1\mod 4}} \bigg( 1 - \frac{1}{p} \bigg)^{-2^{k+1}} \prod_{p|a} \bigg( 1 + \frac{1+\chi_4(p)}{p} \bigg) \bigg( 1 + \frac{2^k}{p}\bigg) \\
&\ll_k  \frac{\sqrt{x}}{\log x} (\log y)^{2^k}\prod_{p|a} \bigg( 1 + \frac{2^{k+1}}{p} \bigg)
\end{align*}
where we completed the product over $p\leq \sqrt{x}$ by adding the missing terms and $p\leq y$ since the terms are $\geq 1$ to get the third line. We also bounded $(1+\chi_4(p))\leq 2$ for the product over $p|a$. We also used $\eqref{Merten_estimate}$ to calculate the products over $p\leq y$ and $p\leq \sqrt{x}$. \\
\noindent Now, we sum this over $a$ to get
$$\ll_k \sum_{\substack{a\leq \sqrt{x} \\ d|a}} \frac{\sqrt{x}}{\log x} (\log y)^{2^{k}} \prod_{p|a}\bigg( 1 + \frac{2^{k+1}}{p}\bigg) .$$
We then have
\begin{align*}
\sum_{\substack{a\leq \sqrt{x} \\ d|a}}\prod_{p|a}\bigg( 1 + \frac{2^{k+1}}{p}\bigg) &= \sum_{\substack{a\leq \sqrt{x} \\ d|a}} \sum_{\ell|a} \frac{\mu^2(\ell) (2^{k+1})^{\omega (\ell)}}{\ell} \\
&= \sum_{\ell} \frac{\mu^2(\ell) (2^{k+1})^{\omega(\ell)}}{\ell} \sum_{\substack{a\leq \sqrt{x} \\ \lcm(\ell , d)|a}} 1 \\
&\leq \sum_{\ell} \frac{\mu^2(\ell)(2^{k+1})^{\omega(\ell)}}{\ell} \frac{\sqrt{x}}{\lcm(\ell , d)} \\
&\leq \frac{\sqrt{x}}{d^{\frac{1}{2}}} \sum_{\ell} \frac{\mu^2(\ell)(2^{k+1})^{\omega(\ell)}}{\ell^{\frac{3}{2}}} \\
&= \frac{\sqrt{x}}{d^{\frac{1}{2}}}  \prod_{p} \bigg(1+ \frac{2^{k+1}}{p^{\frac{3}{2}}} \bigg) \ll_k  \frac{\sqrt{x}}{d^{\frac{1}{2}}},
\end{align*}
where $\lcm(\cdot, \cdot)$ is the least common multiple, and we used $\lcm(\ell , d )\geq \sqrt{\ell d}$. Hence
$$ \sum_{\substack{a\leq \sqrt{x} \\ d| a}} \sum_{\substack{c_0 \sqrt{a}\leq b\leq \sqrt{x} \\ P^+(a^2+b^2)\leq y}} r_0^k(a^2+b^2) \ll_k \frac{x}{\log x} \frac{(\log y)^{2^k}}{d^{\frac{1}{2}}}. $$
\noindent For the remaining sum by omitting the condition $P^+(a^2+b^2)\leq y$ and using $r_0(n)\ll_{\varepsilon} n^{\varepsilon}$, we get
\begin{align*}
\sum_{\substack{a\leq \sqrt{x} \\ d|a}} \sum_{\substack{b\leq c_0 \sqrt{a} \\ P^+(a^2+b^2)\leq y}} r_0^k(a^2+b^2) &\ll_k \sum_{\substack{a\leq \sqrt{x} \\ d|a}} \sum_{b\leq c_0 \sqrt{a} } (a^2+b^2)^{\frac{1}{100}} \\
&\leq \sum_{\substack{a\leq \sqrt{x} \\ d|a}} a^{\frac{1}{50}}\sum_{b\leq c_0 \sqrt{a}} b^{\frac{1}{50}} \\
&\ll \sum_{\substack{a\leq \sqrt{x} \\ d|a}} a^{\frac{53}{100}} \\
&= d^{\frac{53}{100}} \sum_{m\leq \frac{\sqrt{x}}{d}} m^{\frac{53}{100}} \\
&\ll d^{\frac{53}{100}}\bigg(\frac{\sqrt{x}}{d}\bigg)^{\frac{153}{100}} = \frac{x^{\frac{153}{200}}}{d},
\end{align*}
Thus, the contribution from this sum over $a\leq \sqrt{x}$ and $b\leq c_0 \sqrt{a}$ is small.
\end{proof}
\noindent Notice that if $d=1$, then we have the following,
\begin{equation}\label{specific_use}
\sum_{\substack{a^2+b^2\leq x \\ P^+(a^2+b^2)\leq y}} r_0^k(a^2+b^2) = \sum_{\substack{N\leq x \\ P^+(N)\leq y}} r_0^{k+1}(N) \ll_k \frac{x}{\log x}(\log y)^{2^k}
\end{equation}
Now, we use Lemma $\ref{Henriot_use}$ to prove the following corollary.
\begin{corollary}
\label{Henriot}
Let $k\in \mathbb{Z}_{\geq 0}$ and $x$ be a large number. Let $d\in\mathbb{N}$, then
$$ \sum_{\substack{a^2+b^2\leq x \\ d |a}} \frac{r_0^k(a^2+b^2)}{a^2+b^2} \ll_k \frac{(\log x)^{2^k}}{d^{\frac{1}{2}}} . $$
\end{corollary}
\begin{proof}
On taking $y=x=t$ in Lemma $\ref{Henriot_use}$, we get
$$ S(t) = \sum_{\substack{a^2+b^2 \leq t \\ d|a}} r_0^k(a^2+b^2) \ll_k t \frac{(\log t)^{2^k-1}}{d^{\frac{1}{2}}}.$$
Then by partial summation we have
\begin{align*}
\sum_{\substack{a^2+b^2\leq x \\ d |a}} \frac{r_0^k(a^2+b^2)}{a^2+b^2} &= \frac{S(t)}{t} \bigg|_{t=1}^{x} + \int_1^{x} \frac{S(t)}{t^2} dt \\
&\ll \frac{1}{d^{\frac{1}{2}}}\bigg(\int_1^{x} \frac{(\log t)^{2^k-1}}{t} dt + (\log x)^{2^k-1} \bigg) \\
&\ll \frac{(\log x)^{2^k}}{d^{\frac{1}{2}}}. \qedhere
\end{align*}
\end{proof}
For the lower bound in Theorem $\ref{main_result}$, we will need the following two lemmas,
\begin{lemma}\label{lower_bound_r_2}
Let $k\in \Z_{\geq 1}$ and $x$ be a large number. If we have $z\ll x^{\delta}$ for some $0< \delta < 1$. Then
$$ \sum_{\substack{z\leq N\leq x \\ 2|N }} \mu^2(N) \frac{r_0^k(N)}{N} \gg_{k,\delta } (\log x)^{2^{k-1}} . $$
\end{lemma}
\begin{proof}
First we notice that we can get rid of the $2|N$. To see this, let $2m=N$, then
$$ \sum_{\substack{z\leq N\leq x \\ 2|N }} \mu^2(N) \frac{r_0^k(N)}{N} = \frac{1}{2} \sum_{\substack{z \leq 2m \leq x \\ 2\nmid m }} \mu^2(m) \frac{r_0^k(m)}{m} . $$
We use Wirsing's theorem to handle this sum. For a reference, see Koukoulopoulos\cite[Theorem 14.3]{MR3971232}. We will use it with $f(m)= \mu^2(m) r_0^k(m) 1_{2 \nmid m}$. We first note $f(m) \leq \tau^k(m) \leq \tau_{2^k}(m)$ where $\tau_{2^k} = \#\{ n=d_1\cdots d_{2^k} \} $. Moreover,
\begin{align*}
\sum_{p\leq x} \frac{f(p) \log p}{p} &= \sum_{2<p\leq x} \frac{r_0(p)^k \log p}{p} \\
&= 2^k \sum_{\substack{2<p\leq x \\ p\equiv 1 \mod 4}} \frac{\log p}{p} = 2^{k-1} \log x + O_k(1) .
\end{align*}
Thus, Wirsing's theorem implies
$$ \sum_{m\leq x} \frac{f(m)}{m} = \sum_{\substack{m\leq x \\ 2\nmid m}} \frac{\mu^2(m) r_0^k(m)}{m} = \frac{\mathfrak{S}(f)}{\Gamma (2^{k-1}+1)}(\log x)^{2^{k-1}} + O_k( (\log x)^{2^{k-1}-1}) $$
where
$$ \mathfrak{S}(f) = \prod_{p>2} \bigg( 1 - \frac{1}{p} \bigg)^{2^{k-1}} \bigg( 1 +\frac{(1+\chi_4(p))^k}{p} \bigg). $$
Then as $z\leq C x^{\delta}$ for $0<\delta < 1$ we get
\begin{align*}
\sum_{z\leq 2m \leq x} \frac{f(m)}{m} &= \sum_{2m \leq x} \frac{f(m)}{m}  - \sum_{2m \leq z} \frac{f(m)}{m} \\
&= \frac{\mathfrak{S}(f)}{\Gamma (2^{k-1}+1)} \big( (\log x)^{2^{k-1}} - (\log z)^{2^{k-1}} \big) + O_k ((\log x)^{2^{k-1}-1} ) \\
&\geq (\log x)^{2^{k-1}} \bigg(1 -  \bigg(\delta + \frac{\log C}{\log x}\bigg)^{2^{k-1}} \bigg) \gg_{k, \delta} (\log x)^{2^{k-1}}. \qedhere
\end{align*}
\end{proof}
We have sums with admissibility condition. Hence, we need the following lemma to handle such sums.
\begin{lemma}\label{admissible_lower_bound}
Let $k\in \Z_{\geq 2}$, and $\{m_1, n_1,\cdots , m_k , n_k\}\subset \Z$ with $(m_i , n_i)\not=(m_j,n_j)$, for all $i\not=j$, and $\gcd(m_i,n_i)=1$, $\forall i\leq k$. If $m_i^2+n_i^2=m_j^2+n_j^2$ for $i\not=j$ and $2\|(m_i^2+n_i^2)$ for $i\leq k$, then there exists a set $I\subset \{1,2, \cdots , k\}$ such that
\begin{align}
&\#I \gg \frac{k}{\log^2 k} , \nonumber \\
&\forall p\in\CP , \exists a_p\in \F_p : (a_p^2+1)\prod_{i\in I} (m_i a_p + n_i)(n_i a_p - m_i)\not\equiv 0\mod p, \label{admissibility_condition}
\end{align}
which means the vector $(\vect{m},\vect{n})$ restricted to $I$ is an admissible representation.
\end{lemma}
\begin{proof}
First, we notice that the set of solutions to
\begin{equation}\label{zero_equation}
(x^2+1)\prod_{i=1}^k (m_i x + n_i) (n_i x - m_i) \equiv 0 \mod p
\end{equation}
is
$$ x\equiv \pm \sqrt{-1} , \frac{-n_i}{m_i} , \frac{m_i}{n_i} (\bmod p). $$
So, if $p>2k+2$, then we can find some $a_p$, such that $a_p\not\equiv \pm \sqrt{-1}, \frac{-n_i}{m_i}, \frac{m_i}{n_i} (\bmod p)$ for all $i\leq k$. Therefore, we just need to look at $p\leq 2k+2$. We list the primes $p\leq 2k+2$ as
$$ p_1 < p_2 < \cdots < p_{\ell} \leq 2k+2 < p_{\ell+1} .$$ 
Let $I_0 := [k]$ and define recursively
\begin{align*}
S_j(p_u) &= \{ i\in I_{u-1} : p_u| (m_i j + n_i)(n_i j - m_i)\} \\
I_u &= I_{u-1}\setminus S_{j(p_u)}(p_u)
\end{align*}
for $1\leq u\leq \ell$, where $j(p)$ is chosen such that $S_{j(p)}(p)$ is of minimal size. We also want $p\nmid j(p)^2+1$.

\noindent Moreover, noticing for any $i\in I_{u-1}$, we have $i \in S_j(p_u)$ for $j \equiv \frac{-n_i}{m_i} ,\frac{m_i}{n_i}(\bmod p_u )$, we get
$$ (p_u-1- \chi_4(p) ) \# S_{j(p_u)}(p_u) \leq \sum_{\substack{j(\bmod  p_u)  \\ p_u \nmid j^2+1}} \# S_{j}(p_u) \leq \sum_{i\in I_{u-1}} 2 = 2\cdot \# I_{u-1} ,$$
for $2\leq u \leq \ell$. For the remaining $p=2$, we notice since $2\|(m_i^2+n_i^2)$, so $m_i , n_i$ are odd, we have $S_{j(2)}(2)= \emptyset$ where $j(2)=0$.

\noindent Now, we set $I:=I_{\ell}$, this satisfies the admissibility condition $\eqref{admissibility_condition}$ by the definition of $S_j(p)$ for $a_p=j(p)$. Lastly, we have
\begin{align*}
\# I &= \# I_{\ell-1} - \#S_{j(p_{\ell})}(p_{\ell}) \\
&\geq \# I_{\ell-1} \bigg( 1 - \frac{2}{p_{\ell}-1-\chi_4(p)} \bigg) \\
&\geq \# I_0 \prod_{3<p\leq 2k+2} \bigg( 1 - \frac{2}{p-1-\chi_4(p)} \bigg) \gg \frac{k}{\log^2 k} ,
\end{align*}
by Mertens' theorem, since
\begin{equation*}
\frac{1}{p-1-\chi_4(p)} = \frac{1}{p} +  O\bigg(\frac{1}{p^2}\bigg) .  \qedhere
\end{equation*}
\end{proof}
\noindent We will use this in the case $m_i^2+n_i^2$ is squarefree, which satisfies the coprimality condition.

\section{Counting Prime Representations}\label{fourth_section}
In this section we give the main idea of the proof. Fix $k\in \Z_{\geq 2}$.

We will always use $n\in\mathbb{N}$ for the number with $k$ representations as the sum of two prime squares, $n=p_i^2+q_i^2$ for $i\in [k]$. We let $P^+(n)=p'$ be the largest prime factor of $n$ and the number $N=n/p'$ will always be an integer with $k$ representations as the sum of two squares, $N=m_i^2+n_i^2$ for $i\in [k]$. 

\noindent Define also a smoothing variable $y$ as
$$ y= x^{\frac{1}{\log\log x}}.  $$

Let $n\in \mathbb{N}$ be, as we have stated, such that, 
$$ n:=p_1^2+q_1^2 = \cdots = p_k^2+q_k^2 $$
for $(\vect{p} , \vect{q})\in \CH_k(x)$, where $\CH_k(x)$ is defined in $\eqref{S_k_definition}$.

\noindent We see that if $p\equiv 3 \mod 4$ divides $n$, then it must divide each $p_i , q_i$ for $i\in [k]$, but this is not counted in $\CH_k(x)$. Hence, any prime divisor of $n$ must be $1\mod 4$. Define $p':=P^+(n)$, then $p'=r^2+s^2$ for some $r,s\in\mathbb{N}$.

\noindent Moreover, if we let,
$$ N:= \frac{n}{p'} ,$$
then we have $P^+(N)\leq p'$. Now since $r^2+s^2$ is a prime, then $r+ i s$ should be a prime in $\Z[i]$. Then since $r+i s | (p_j^2+q_j^2)= (p_j + i q_j)(p_j - i q_j)$ for every $j\in [k]$, thus it divides one of multiples. Without loss of generality say $r+i s | p_j + i q_j$, and then we get $\frac{p_j + i q_j}{r+i s} = m_j + i n_j$. Thus, we can find $\{m_1 , \cdots , n_k\} \subset \Z$ such that
\begin{align*}
N=m_1^2+n_1^2 &= \cdots = m_k^2+n_k^2 \\
p_i = m_i r - n_i s &\text{ and } q_i = n_i r + m_i s , \forall i\in [k].
\end{align*}
Hence, if we run through $(\vect{m},\vect{n})$ and $(r,s)$, as they run through the values where $m_i r- n_i s$ and $n_i r + m_i s$ are prime, we are counting the representations $(p_i , q_i)$. Thus, we define the following quantities to capture this behaviour,
\begin{align*}
f_k^*(z,(\vect{m}, \vect{n})) &:= \#\{ (r,s)\in \mathcal{C}(z,(\vect{m},\vect{n})) : P^+(N)\leq r^2+s^2  , (\ref{what_we_want}) \text{ holds, }\forall i\in[k] \}, \\
f_k(z,(\vect{m}, \vect{n})) &:= \#\{ (r,s)\in \mathcal{C}(z,(\vect{m},\vect{n})) : (\ref{what_we_want}) \text{ holds, }\forall i\in[k] \} ,
\end{align*}
where
$$ \mathcal{C}(z,(\vect{m},\vect{n})):=\{ (r,s)\in \mathbb{N}^2 : r^2+s^2\leq z , 0 < m_i r - n_i s < n_i r + m_i s  \} .$$
We notice $f_k^*(z, (\vect{m},\vect{n}))\leq f_k(z,  (\vect{m},\vect{n}))$. We also have, by sieve theory, good bounds for $f_k(z,(\vect{m},\vect{n}))$ when $z$ is large.

\begin{remark} \label{N_conditions}
We notice that to make $\eqref{what_we_want}$ prime, any $N=m_1^2+n_1^2 =\cdots = m_k^2+n_k^2$ must satisfy,
$$\gcd(m_i , n_i)=1 , \forall i\in[k], $$
which implies
$$ 2\|N \text{ and } N \in \CM = \{ u\in \N : p|u \implies p=2 \text{ or } p\equiv 1 \mod 4 \} , $$
as otherwise we would have $2|m_i , n_i$ or $p|m_i , n_i$ for some $p\equiv 3\mod 4$, which is impossible as $m_i r - n_i s$ can't be prime. So, we can use Lemma $\ref{r_2_inequality}$ to split $r_0$ over divisors for an upper bound. 
\end{remark}
\noindent We also notice that in the definition of $\CH_k(x)$ in $\eqref{S_k_definition}$, we have the condition $(p_i , q_i)\not=(p_j , q_j)$. Since $\gcd(r,s)=1$, $(p_i,q_i)=(p_j,q_j)$ is equivalent to
\begin{align*}
\begin{array}{ll} m_i r - n_i s= m_j r - n_j s \\ n_i r + m_i s = n_j r + m_j s \end{array} &\iff \begin{array}{ll} r(m_i - m_j) = s(n_j - n_i) \\ r(n_i - n_j) = s(m_j - m_i) \end{array} \\
&\implies r|(n_j - n_i) \text{ and } r^2 | (m_j - m_i) \\
&\implies r^3 | (n_j - n_i) \text{ and } r^4 | (m_j - m_i) \\
&\implies r^b | (n_j - n_i) \text{ and } r^b | (m_j - m_i) , \forall b\geq 1 ,
\end{align*}
similarly for $s$, and since $r$ and $s$ can't both be 1, we get that $(m_i , n_i)=(m_j , n_j)$. Therefore, we define the set of $N$,
$$ \CA (x) := \{ (\vect{m} , \vect{n}) \in \Z^{2k} : m_i^2+n_i^2 \leq x , \forall i\in[k] ,  m_i^2+n_i^2=m_j^2+n_j^2, (m_i,n_i)\not=(m_j,n_j), \forall i\not=j \} .$$
We will need the following equality that we get by $\eqref{number_of_representations}$,
\begin{equation} \label{representations_to_N}
\sum_{\substack{(\vect{m},\vect{n})\in \CA(N) \\ N=m_i^2+n_i^2, \forall i\in [k]}} 1 = k! \binom{4r_0(N)}{k} .
\end{equation}
For the upper bound, when we run through $(\vect{m},\vect{n})\in \CA (x)$ summing $f_k^*(\frac{x}{N}, (\vect{m},\vect{n}) )$, we are overcounting. For this, first we split our sum $\CS_k$ (see the definition in $\eqref{the_single_sum}$) into a $y$-smooth and the non-smooth part. Then for the non-smooth part, we can write it as a sum over $(\vect{m},\vect{n})$ and $(r,s)$, where $r^2+s^2=p'> y$. We further split this into when $y < p' \leq x^{\frac{1}{3}}$ and $x^{\frac{1}{3}} < p' \leq x$. Then we get
$$\CS_k = \sum_{(\vect{p}, \vect{q})\in \CH_k(x)} 1 = \sum_{\substack{(\vect{p}, \vect{q})\in \CH_k(x) \\ P^+(n)\leq y}}1 + \sum_{\substack{(\vect{p}, \vect{q})\in \CH_k(x) \\ y< P^+(n)\leq x^{\frac{1}{3}}}}1 + \sum_{\substack{(\vect{p}, \vect{q})\in \CH_k(x) \\ x^{\frac{1}{3}}<P^+(n)\leq x}}1 .$$
We can use the function $f_k^*(z, (\vect{m},\vect{n}) )$ to bound these sums as
\begin{align}
\leq \sum_{\substack{(\vect{p},\vect{q})\in \CH_k (x) \\  P^+(n)\leq y}} 1 &+ \sum_{\substack{(\vect{m}, \vect{n}) \in \CA(\frac{x}{y}) \\ x^{\frac{2}{3}}\leq N \leq \frac{x}{y}}} \bigg(f_k^*\bigg(\frac{x}{N}, (\vect{m},\vect{n})  \bigg) - f_k^*(y, (\vect{m},\vect{n}) ) \bigg) \label{main_equation} \\
&+ \sum_{(\vect{m},\vect{n})\in \CA(x^{\frac{2}{3}})} \bigg( f_k^* \bigg( \frac{x}{N} , (\vect{m},\vect{n}) \bigg) - f_k^*(x^{\frac{1}{3}}, (\vect{m},\vect{n}) ) \bigg) .\nonumber
\end{align}
Remember that we also have $P^+(N)\leq P^+(n) \leq p'$ in the function $f_k^*$. We bound $f_k^*$ by $f_k$ in the third sum, but keep $f_k^*$ in the second one. Since $m_i^2+n_i^2=N=\frac{n}{p'} \leq \frac{x}{p'}$, we have the bounds on $N=m_i^2+n_i^2$ in $\eqref{main_equation}$.

\noindent For the lower bound, we need to undercount. For this, we restrict our $n$ to be squarefree. Then $N$ is also squarefree and we have $\gcd(N,p')=1$, hence $P^+(N)<p'$ and
$$ r_0(n)=r_0(N)r_0(p') . $$
Hence, when summing $f_k^*(x/N, (\vect{m},\vect{n}) )$ over $(\vect{m},\vect{n})\in \CA(x)$, we count all the representations provided that $P^+(N)<p'$. On restricting our conditions $N\leq x$ and $p' \leq \frac{x}{N}$ to $x^{\frac{1}{3}} \leq N \leq x^{\frac{3}{8}}$ and $x^{\frac{1}{2}} \leq p' \leq \frac{x}{N}$, we can get rid of the condition $P^+(N) < p'$, and get
\begin{equation}\label{rough_lower_bound}
\CS_k \geq \sum_{\substack{(\vect{m} , \vect{n})\in \CA(x^{\frac{3}{8}}) \\ x^{\frac{1}{3}} \leq N \leq x^{\frac{3}{8}} \\ N \text{ squarefree} }} \bigg( f_k \bigg( \frac{x}{N} ,  (\vect{m},\vect{n}) \bigg) - f_k(x^{\frac{1}{2}},  (\vect{m},\vect{n})) \bigg).
\end{equation}

\noindent Using the Selberg sieve, we will prove:
\begin{lemma}
Let $z$ be a sufficiently large number. Assume that $m_i , n_i \in \Z $ with $\gcd(m_i , n_i)=1$ for $i\in [k]$. Then we have the following bound
\begin{equation} \label{sieve_bound}
f_k(z,  (\vect{m},\vect{n})) \ll_k \mathfrak{S}(m_1, n_1 , \cdots , m_k , n_k) \frac{z}{\log^{2k+1} z},
\end{equation}
where
\begin{equation}\label{singular_series_def}
\mathfrak{S}(m_1, n_1 , \cdots , m_k , n_k) = \prod_p \bigg(1 - \frac{\nu_p(m_1, \cdots , n_k)}{p^2} \bigg) \bigg(1 - \frac{1}{p} \bigg)^{-(2k+1)} , 
\end{equation}
and
$$ \nu_p(m_1, n_1, \cdots , m_k , n_k) = \#\Big\{ (r,s)\in\mathbb{F}_p^2 : (r^2+s^2)\prod_{i=1}^k (m_i r - n_i s)(n_i r + m_i s) =0 \Big\} . $$
\end{lemma}
\begin{proof}
Let $P:=P(z^{\frac{1}{2k+1}})$ where $P(B) := \prod_{p\leq B} p$. We have
$$ f_k(z, (\vect{m},\vect{n})) \leq \sum_{\substack{|r|,|s|\leq \sqrt{z} \\ (\ref{what_we_want}) \text{ holds}}} 1 \leq \sum_{|r|,|s|\leq \sqrt{z}} \bigg( \sum_{\substack{ d| P \\ d|(r^2+s^2)\prod_{i=1}^k (m_i r - n_i s)(n_i r + m_i s)}} \lambda_d \bigg)^2 ,$$
where $\lambda_d$ are real numbers with the condition $\lambda_1=1$.

\noindent We let $\nu, \mu \mod d$ be solutions to $(\nu^2+\mu^2)\prod_{i=1}^k(m_i \nu - n_i \mu)(n_i \nu + m_i \mu) \equiv 0 \mod d$, then for each solution, we are counting $|r|,|s|\leq \sqrt{z}$ such that $r\equiv \nu \mod d$ and $s\equiv \mu \mod d$. Thence we get, for $d|P$
\begin{align*}
\CA_d = \sum_{\substack{|r|,|s| \leq \sqrt{z} \\ d| (r^2+s^2)\prod_{i=1}^k (m_i r - n_i s)(n_i r + m_i s)}} 1 &= \sum_{\substack{\nu , \mu (\bmod d) \\ (\nu^2+\mu^2)\prod_{i=1}^k(m_i \nu - n_i \mu)(n_i \nu + m_i \mu) \equiv 0 (\bmod d)}} \sum_{\substack{|r|,|s|\leq \sqrt{z} \\ r\equiv \nu (\bmod d) \\ s\equiv \mu (\bmod d)}} 1 \\
&= g(d) 4 z + O (g(d) d \sqrt{z} ) = 4\sqrt{z} \bigg(  g(d) \sqrt{z} + O (g(d) d ) \bigg),
\end{align*}
where $g(d)=\frac{\nu(d)}{d^2}$ is a multiplicative function with 
$$\nu(D) = \#\{ (r,s)\in (\Z/D\Z)^2 : (r^2+s^2)\prod_{i=1}^k(m_i r - n_i s)(n_i r + m_i s) \equiv 0 \mod D \} .$$
Notice that $\nu (D) \leq D^2$ and $\nu(p)=\nu_p(m_1,\cdots ,n_k)$. Now, let $p|P$, we see that by Remark $\ref{N_conditions}$, since $\gcd(m_1, n_1)=1$, $p$ can't divide both $m_1$ and $n_1$, without loss of generality assume $\gcd(m_1 , p)=1$. Then $(n_1 : m_1)$ is a projective solution of
$$ (r^2+s^2)\prod_{i=1}^k (m_i r - n_i s)(n_i r + m_i s) = 0  $$
in $\mathbb{P}(\mathbb{F}_p)$. This gives $p$ many solutions in $\F_p^2$, hence, we have $\nu(p)\geq p$, which implies $g(d)d = \nu(d) / d \geq 1$ for any $d|P$, since $g(d)$ is multiplicative and $P$ is squarefree. We also have
$$ \sum_{w\leq  p\leq x} g(p) \log p = \sum_{w\leq p \leq x} \frac{\nu_p(m_1, \cdots ,n_k)}{p^2}\log p \ll_k \sum_{w\leq p\leq x} \frac{\log p}{p} \ll \log (2x/w) $$
because $(r^2+s^2) \prod_{i=1}^k(m_i r - n_i s)(n_i r + m_i s)$ has at most $(2k+2)p+1$ solutions $\bmod p$. We will see more details on $\nu_p(m_1, \cdots , n_k)$ in Lemma $\ref{omega_p}$.

\noindent Therefore, we get the result by applying the Selberg sieve, see Iwaniec-Kowalski \cite[Theorem 6.5]{MR2061214}.
\end{proof}
We want to understand the singular series $\mathfrak{S}(m_1 , \cdots  , n_k)$. To this end, we define
\begin{equation} \label{Chaotic_Product}
R(\vect{m}, \vect{n})= \prod_{i<j} R(m_i,m_j, n_i , n_j),
\end{equation}
where
$$ R(m_i , m_j , n_i , n_j)= (m_i n_j - m_j n_i)(m_i m_j+n_i n_j) .$$
We have the following lemma on $\nu_p$.
\begin{lemma}
\label{omega_p}
Assume $m_i , n_i \in \Z$ with $\gcd(m_i , n_i)=1$ for $i\in [k]$ and $N=m_1^2+n_1^2 = \cdots = m_k^2+n_k^2$. Then we have
$$ \nu_p(m_1 , \cdots , n_k) = q_p(p-1)+1 , $$
where
$$ q_p = \left\{ \begin{array}{lcr} 1+\chi_4(p), & p|N  \\ 2k+1+\chi_4(p), & p\nmid N \cdot R(\vect{m}, \vect{n}) \end{array} \right.$$
and 
$$3+\chi_4(p) \leq q_p < 2k+1+\chi_4(p) \text{ for } p\nmid N, \text{ but } p| R(\vect{m}, \vect{n}) . $$
Moreover, if $(\vect{m},\vect{n})$ is an admissible representation, then 
$$q_p\leq p \, \text{  and  }\, \mathfrak{S} (m_1 , \cdots , n_k)\not=0 .$$
\end{lemma}
\begin{proof}
By looking at 
\begin{equation}\label{equation_roots}
(r^2+s^2)\prod_{i=1}^k (m_i r - n_i s )(n_i r + m_i s)=0
\end{equation}
in $\mathbb{P}(\F_p)$, we can write
$$\nu_p = q_p (p-1)+1.$$
since $\gcd(m_i , n_i)=1$ for $i\in [k]$ and the equation is homogeneous. 

\noindent \textsf{Case 1.}
Assume $p|N$. We know that the non-zero solutions of $\eqref{equation_roots}$ belong to
\begin{equation} \label{solutions}
\{(1 : \sqrt{-1}),(-1 : \sqrt{-1}) , (n_1 :  m_1), (- m_1 : n_1), \cdots , (n_k : m_k) , (-m_k: n_k)\}
\end{equation}
in the projective space over $\mathbb{F}_p$. Now, since $p\nmid m_i$, $p\nmid n_i$, then the solutions $(-m_i : n_i) , (n_i :  m_i)$ give
$$ r^2+s^2 \equiv t^2 (m_i^2+n_i^2) \equiv 0 \mod p $$
for some $t\not=0$. So all the solutions from $\eqref{solutions}$ are the solutions of $r^2+s^2 =0$ in $\mathbb{P}(\F_p)$.

\noindent Hence, the set of solutions is only $\{(1:\sqrt{-1}), (-1: \sqrt{-1})\}$. But this gives $1+\chi_4(p)$ many solutions in the projective space over $\F_p$, depending on whether $p\equiv 1\mod 4$ or $3\mod 4$. Thus, $q_p= 1+\chi_4(p)$. \\
\textsf{Case 2.}
Assume $p\nmid N \cdot R(\vect{m},\vect{n})$. Then the elements $(-m_i : n_i )$ and $(n_j : m_j)$ of the set $\eqref{solutions}$ are distinct since $p\nmid (m_i m_j + n_i n_j)$ and similarly $(-m_i : n_i)$ and $(-m_j : n_j)$. By a similar argument to case 1, they are different from $(\pm 1 : \sqrt{-1})$. Hence, in total we have $q_p = 2k+1+\chi_4(p)$. \\
\textsf{Case 3.}
Assume $p\nmid N$, but $p| R(\vect{m}, \vect{n})$. Then we have some of the $(-m_i : n_i)$ or $(n_i : m_i)$ in $\eqref{solutions}$ as the same solutions. But $p$ can't divide both $m_i n_j - m_j n_i$ and $m_i m_j + n_i n_j$ since otherwise $p|N$. Then we expect the solutions $\{(1:\sqrt{-1}) , (-1 : \sqrt{-1})\}$ and also at least two more solution from $\eqref{solutions}$. Hence, $q_p \geq 3 +\chi_4(p)$. The upper bound is trivial, since we can't have more solutions than the set $\eqref{solutions}$ and not all of them are distinct.

\noindent Lastly, if $(\vect{m},\vect{n})$ is an admissible representation, then we know $\mathfrak{S}(m_1,\cdots , n_k)\not=0$ by Definition $\ref{admissibility}$. Now, since $q_p\leq p+1$, and $\nu_p(m_1,\cdots , n_k) = q_p(p-1)+1 = p^2$ whenever $q_p=p+1$, we get that $q_p\leq p$, because $\mathfrak{S}(m_1,\cdots ,n_k)\not=0$.
\end{proof}
\noindent By the above lemma, for an admissible representation $(\vect{m},\vect{n})$ the singular series $\eqref{singular_series_def}$ can be rewritten as
$$ \mathfrak{S}(m_1, \cdots ,n_k) = \prod_p \bigg( 1- \frac{q_p-1}{p} \bigg) \bigg(1 -  \frac{1}{p} \bigg)^{-2k} \asymp_k \prod_p \bigg( 1 - \frac{q_p-1}{p}\bigg) \bigg(1 + \frac{2k}{p} \bigg) $$
since $\prod_p (1-\frac{1}{p})^{-2k} (1 + \frac{2k}{p})^{-1} \asymp_k 1$. By multiplying the factors of the product and bounding the contribution of the $p^2$ factor we get
$$ \mathfrak{S}(m_1 , \cdots , n_k) \asymp_k \prod_p \bigg(1 + \frac{2k+1-q_p}{p} \bigg) . $$
We can multiply and divide by $L(1, \chi_4)\asymp \prod_p (1+\frac{\chi_4(p)}{p})$ to get
$$\mathfrak{S}(m_1 , \cdots , n_k) \asymp_k \prod_p \bigg(1+\frac{\chi_4(p)}{p} \bigg) \bigg(1 + \frac{2k+1-q_p}{p} \bigg) .$$
We can multiply and get rid of terms with $1/p^2$ again to obtain
$$ \mathfrak{S}(m_1, \cdots ,n_k) \asymp_k \prod_p \bigg(1 + \frac{2k+1+\chi_4(p)-q_p}{p} \bigg) .$$
By Lemma $\ref{omega_p}$, $q_p\leq 2k+1+\chi_4(p)$, thus,
\begin{equation} \label{S_lower_bound}
\mathfrak{S}(m_1 , \cdots , n_k) \gg_k 1.
\end{equation}
Also by Lemma $\ref{omega_p}$, $q_p \geq 1+\chi_4(p) $ for $p|N \cdot R(\vect{m},\vect{n})$ and $q_p= 2k+1+\chi_4(p)$ otherwise. Hence, we have
$$ \mathfrak{S}(m_1, \cdots ,n_k) \ll_k \prod_{p|N \cdot R(\vect{m}, \vect{n})} \bigg(1+ \frac{2k}{p} \bigg) $$
Applying AM-GM inequality and getting rid of $1/p^2$ terms, we have
\begin{align*}
 \mathfrak{S}(m_1 , \cdots ,n_k) &\ll_k \prod_{p|N} \bigg( 1+\frac{2k}{p}\bigg)^2 + \prod_{\substack{p| R(\vect{m} , \vect{n}) \\ p\nmid N}} \bigg( 1+\frac{2k}{p} \bigg)^2 \\
&\asymp_k \prod_{p|N} \bigg(1+\frac{4k}{p} \bigg) + \prod_{\substack{p|R(\vect{m}, \vect{n}) \\ p\nmid N}} \bigg(1+\frac{4k}{p} \bigg)
\end{align*}
since we can bound $\prod_p ( 1 +\frac{2k}{p})^2 (1+\frac{4k}{p})^{-1} \asymp_k 1$. Then turning the products into sums we get,
\begin{equation}
\label{Singular_Series} \mathfrak{S} (m_1 ,\cdots , n_k) \ll_k \sum_{\ell_1 | N} \frac{\mu^2(\ell_1) (4k)^{\omega(\ell_1)}}{\ell_1} + \sum_{\substack{\ell_2 |R(\vect{m},\vect{n}) \\ \gcd(\ell_2 , N)=1 }} \frac{\mu^2(\ell_2) (4k)^{\omega(\ell_2)}}{\ell_2} 
\end{equation}
for $(\vect{m}, \vect{n})$ an admissible representation. To use this bound for the singular series, we need the admissibility condition. But notice if we don't have an admissible representation, then the singular series is $0$, thus $\eqref{Singular_Series}$ applies whether or not the representation is admissible.

\noindent We can return to our sum $\CS_k$. For bounds on $S_k$, we have by $\eqref{main_equation}$ and $\eqref{rough_lower_bound}$,
\begin{align}
&\label{upper_bound_equation}\CS_k \leq \sum_{(\vect{m} , \vect{n}) \in \CA (x^{\frac{2}{3}})} f_k\bigg(\frac{x}{N} ,  (\vect{m},\vect{n})\bigg) + \sum_{\substack{(\vect{m} , \vect{n}) \in \CA (\frac{x}{y}) \\ x^{\frac{2}{3}} \leq N \leq \frac{x}{y}}} f_k^*\bigg(\frac{x}{N},  (\vect{m},\vect{n}) \bigg) + \sum_{\substack{n\leq x \\ p|n \implies p\leq y}} r_2^k(n) \\
&\label{lower_bound_equation}\CS_k \geq \sum_{\substack{(\vect{m} , \vect{n})\in \CA(x^{\frac{3}{8}}) \\ x^{\frac{1}{3}} \leq N \leq x^{\frac{3}{8}} \\ N \text{ squarefree} }} \bigg( f_k \bigg( \frac{x}{N},  (\vect{m},\vect{n}) \bigg) - f_k(x^{\frac{1}{2}} ,  (\vect{m},\vect{n})) \bigg). 
\end{align}
where we have used the trivial bounds for the upper bound.

\noindent For the upper bound, we show that we can take the smooth part, i.e. the last term, as small as we want. For the non-smooth parts, we first expand the second sum $e$-adically and then split the singular series via $\eqref{Singular_Series}$. Then for the sum with the condition 
$$ \ell_2 | R(\vect{m} , \vect{n}) , $$
in $\eqref{Singular_Series}$, we separate the divisors using $\eqref{Chaotic_Product}$ to give
$$ d_{\{ i,j\}} | R(m_i, m_j , n_i , n_j) \text{ for } i<j\leq k , $$
where $\ell_2 = \prod_{i<j} d_{\{i,j\}}$. We call this procedure condition separation and use it to further simplify summation conditions.

\noindent For the lower bound, we apply Conjecture $\ref{conjectural_lower_bound}$ with $\eqref{S_lower_bound}$ to get a sum over admissible representations and then use Lemma $\ref{admissible_lower_bound}$ to get a sum over a subset without the admissibility condition that we can handle.

\section{Smooth Part of $\CS_k$}
For any $n\in \CM$ (see Lemma $\eqref{r_2_inequality}$), we have $r_2(n) \leq r_0(n)$ and $r_2(n)=0$ otherwise, so, by Rankin's trick with $\sigma = C_k \frac{\log\log x}{\log x}$ for $C_k \geq 2^{k-1}+4$, we get
\begin{align*}
\sum_{\substack{n\leq x \\ p|n\implies p\leq y}} r_2^k(n) &\leq \sum_{\substack{n\leq x \\ n\in \CM \\ p|n\implies p\leq y}} r_0^k(n) \\
&\leq x^{1-\sigma} \sum_{\substack{n\leq x \\ n\in \CM \\ p|n \implies p\leq y}} \frac{r_0^k(n)}{n^{1-\sigma}} \\
&\ll_k x^{1-\sigma} \prod_{\substack{p\leq y \\ p\equiv 1 (\bmod 4)}} \bigg( 1 + \frac{2^k}{p^{1-\sigma}} + O_k\bigg( \frac{1}{p^{2(1-\sigma)}} \bigg) \bigg) \\
&\ll_k x^{1-\sigma} \exp \bigg( \sum_{\substack{p\leq y \\ p\equiv 1 (\bmod 4)}} \frac{2^k}{p^{1-\sigma}} \bigg),
\end{align*}
where to get the last inequality, we took $x$ large enough so that $1-\sigma > \frac{3}{4}$, thus higher order terms contribute $\ll_k 1$. Since $p\leq y$, we can use the Mean Value Theorem to get
$$ \bigg| \frac{1}{p} - \frac{1}{p^{1-\sigma}} \bigg| \leq \sigma \frac{\log p}{p^{1-\sigma}} \leq y^{\sigma}\sigma \frac{\log p}{p} . $$

\noindent Inserting this in the above we obtain
\begin{align*}
\sum_{\substack{n\leq x \\ p|n\implies p\leq y}} r_2^k(n)&\ll_k \frac{x}{\log^{C_k}x} \exp \bigg( 2^k \sum_{\substack{p\leq y \\ p\equiv 1(\bmod 4)}} \frac{1}{p} + O_k \bigg(y^{C_k\frac{\log\log x}{\log x}} \frac{\log\log x}{\log x}\sum_{p\leq y} \frac{\log p}{p} \bigg) \bigg) \\
&= \frac{x}{\log^{C_k}x} \exp\bigg( 2^k \sum_{\substack{p\leq y \\ p\equiv 1 (\bmod 4)}} \frac{1}{p} + O_k (1)  \bigg) \ll_k \frac{x}{\log^4 x} .
\end{align*}
\section{Non-Smooth Part of $\CS_k$}\label{sixth_section}
Now we look at the main sum in $\eqref{upper_bound_equation}$
\begin{equation}\label{upper_bound_reference}
\sum_{(\vect{m}, \vect{n})\in \CA (x^{\frac{2}{3}} )} f_k \bigg( \frac{x}{N} ,  (\vect{m},\vect{n})\bigg) + \sum_{\substack{(\vect{m}, \vect{n})\in \CA(\frac{x}{y}) \\ x^{\frac{2}{3}} \leq N \leq \frac{x}{y}}} f_k^* \bigg( \frac{x}{N},  (\vect{m},\vect{n}) \bigg) .
\end{equation}
Now, we expand $e$-adically the $f_k^*$ in the second sum
$$ \sum_{\substack{(\vect{m}, \vect{n})\in \CA (\frac{x}{y}) \\ x^{\frac{2}{3}} \leq N \leq \frac{x}{y} }} \sum_{ \log y \leq \ell \leq \log x/N} (f_k^*(e^{\ell+1},  (\vect{m},\vect{n})) - f_k^*(e^{\ell},  (\vect{m},\vect{n})) ) \leq \sum_{\log y \leq \ell \leq \log x^{1/3}} \sum_{\substack{(\vect{m}, \vect{n}) \in \CA (\frac{x}{e^{\ell}}) \\ P^+(N)\leq e^{\ell+1} }} f_k(e^{\ell+1},  (\vect{m},\vect{n}))  ,$$
since $N\leq \frac{x}{p'} \leq \frac{x}{e^{\ell}}$, $P^+(N)\leq p'\leq e^{\ell+1}$, and $f_k^*(z,  (\vect{m},\vect{n}))\leq f_k(z,  (\vect{m},\vect{n}))$.
Now, we use our sieve theory bound $\eqref{sieve_bound}$ to bound the sum in $\eqref{upper_bound_reference}$ by
\begin{equation}\label{main_upper_bound}
\ll_k x \sum_{\substack{(\vect{m}, \vect{n})\in \CA (x^{\frac{2}{3}} )}} \frac{\mathfrak{S}(m_1, \cdots , n_k)}{N(\log x/N)^{2k+1}} + \sum_{\log y \leq \ell \leq \log x^{1/3}} \frac{e^{\ell}}{\ell^{2k+1}} \sum_{\substack{(\vect{m}, \vect{n})\in \CA(\frac{x}{e^{\ell}}) \\ P^+(N)\leq e^{\ell+1}}} \mathfrak{S}(m_1 , \cdots ,n_k) .
\end{equation}
We start with the first sum, and bound the second one in Section $\ref{dyadic_sum}$. As $N\leq x^{\frac{2}{3}}$ in the first summand above we can take out the $(\log x/N)^{-(2k+1)}$ as $(\log x)^{-(2k+1)}$. This gives us
$$ \ll_k  \frac{x}{\log^{2k+1} x} (M_1 + M_2), $$
where
\begin{align*}
M_1 &= \sum_{(\vect{m}, \vect{n})\in \CA (x^{\frac{2}{3}} )} \frac{1}{N}\sum_{\ell_1 | N} \frac{\mu^2(\ell_1) (4k)^{\omega(\ell_1)}}{\ell_1} \\
M_2 &= \sum_{(\vect{m}, \vect{n})\in \CA (x^{\frac{2}{3}} )} \frac{1}{N}\sum_{\substack{\ell_2 | R(\vect{m},\vect{n}) \\ \gcd(\ell_2 , N)=1}} \frac{\mu^2(\ell_2) (4k)^{\omega(\ell_2)}}{\ell_2} .
\end{align*}
\noindent Our aim is to show that $M_i \ll_k (\log x)^{2^{k-1}}$ for $i=1,2$. \\
We start with $M_1$ and change the order of summation,
$$ M_1 = \sum_{\ell_1} \frac{\mu^2(\ell_1)(4k)^{\omega(\ell_1)}}{\ell_1} \sum_{\substack{(\vect{m},\vect{n})\in \CA(x^{\frac{2}{3}}) \\ \ell_1 | N }} \frac{1}{N} ,$$
but the inner sum is equal to
\begin{equation}\label{bound_for_representations}
\sum_{\substack{N\leq x^{\frac{2}{3}} \\ \ell_1 |N}} \frac{1}{N} \sum_{\substack{(\vect{m}, \vect{n})\in \CA(N) \\ N=m_i^2+n_i^2, \forall i\in[k]}} 1 \leq \sum_{\substack{N\leq x^{\frac{2}{3}} \\ \ell_1 | N}} \frac{(4r_0(N))^k}{N} \ll_k \sum_{\substack{N\leq x^{\frac{2}{3}} \\ \ell_1 | N}} \frac{r_0^k(N)}{N}  ,
\end{equation}
by $\eqref{representations_to_N}$. Thus, using this, we get
\begin{align*}
M_1 &\ll_k \sum_{\ell_1} \frac{\mu^2(\ell_1) (4k)^{\omega(\ell_1)} r_0^k(\ell_1)}{\ell_1^2} \sum_{m\leq \frac{x^{\frac{2}{3}}}{\ell_1}} \frac{r_0^k(m)}{m}  \\
&\ll_k \sum_{\ell_1} \frac{\mu^2(\ell_1) (4k)^{\omega(\ell_1)} r_0^k(\ell_1)}{\ell_1^2} \bigg(\log \frac{x^{\frac{2}{3}}}{\ell_1 }\bigg)^{2^{k-1}}  \\
&\ll_k (\log x)^{2^{k-1}} \prod_{p\equiv 1(\bmod 4)} \bigg( 1+\frac{2^k 4k}{p^2} \bigg)  \ll_k (\log x)^{2^{k-1}} , 
\end{align*}
where we use Lemma \ref{r_2_inequality} in the first line with $N=m\ell_1$ and Lemma \ref{r_2_bound} for the second line.
\section{Upper Bound for $M_2$ and Condition Separation}\label{seventh_section}
In this section, we bound $M_2$ using condition separation. To do this, we first let $K=\binom{k}{2}$ and $X= x^{\frac{2}{3}}$. Then we have, by condition separation
\begin{align*}
M_2 &= \sum_{\ell_2} \frac{\mu^2(\ell_2) (4k)^{\omega(\ell_2)}}{\ell_2} \sum_{\substack{(\vect{m}, \vect{n})\in \CA (X ) \\ \ell_2 | R(\vect{m},\vect{n}) \\ \gcd(\ell_2 , N)=1}} \frac{1}{N} \\
&\leq \sum_{\ell_2} \frac{\mu^2(\ell_2) (4k)^{\omega(\ell_2)}}{\ell_2} \sum_{\prod_{1\leq i<j\leq k}d_{\{i,j\}} = \ell_2} \bigg( \prod_{1\leq i<j \leq k} M_{d_{\{i,j\}}} \bigg)^{\frac{1}{K}} ,
\end{align*}
where
$$ M_{d_{\{i,j\}}} := \sum_{\substack{(\vect{m}, \vect{n})\in \CA ( X ) \\ d_{\{i,j\}} | R(m_i , m_j, n_i , n_j)  \\ \gcd(d_{\{i,j\}} , N)=1}} \frac{1}{N} $$
We are going to use this idea a couple of times to handle the multiple conditions at the same time. \\
\noindent Our goal is to show $M_{d_{\{i,j\}}} \ll_k (\log x)^{2^{k-1}-1}/ d_{\{i,j\}}^{1/32} $. Without loss of generality we can only look at the sum $M_{d_{\{1,2\}}}$. We also notice that $(m_3 , n_3) , \cdots , (m_k , n_k)$ are free of any conditions, so we can bound them like we did in $\eqref{bound_for_representations}$, and use condition separation again to get
\begin{align}
M_{d_{\{1,2\}}} &\leq \sum_{d_{\{1,2\}}= d_1 d_2} \bigg( \sum_{\substack{(\vect{m}, \vect{n})\in \CA ( X ) \\ d_1 | (m_1 n_2- m_2 n_1) \\ \gcd(d_1 , N)=1}} \frac{1}{N} \bigg)^{\frac{1}{2}} \bigg( \sum_{\substack{(\vect{m}, \vect{n})\in \CA ( X ) \\ d_2 | (m_1 m_2+n_1 n_2) \\ \gcd(d_2 , N)=1}} \frac{1}{N} \bigg)^{\frac{1}{2}} \nonumber \\
&\ll_k \sum_{d_{\{1,2\}}= d_1 d_2} \bigg( \sum_{\substack{ m_1^2+n_1^2 = m_2^2+n_2^2 \leq X \\ d_1 | (m_1 n_2- m_2 n_1) \\ \gcd(d_1 , N)=1}} \frac{r_0^{k-2}(N)}{N} \bigg)^{\frac{1}{2}} \bigg( \sum_{\substack{m_1^2+n_1^2=m_2^2+n_2^2\leq X  \\ d_2 | (m_1 m_2+n_1 n_2) \\ \gcd(d_2 , N)=1}} \frac{r_0^{k-2}(N)}{N} \bigg)^{\frac{1}{2}} . \label{inner_sum_reference}
\end{align}
Now, we take a look at $N=m_1^2+n_1^2=m_2^2+n_2^2$ in $\mathbb{Z}[i]$. Then $m_1^2+n_1^2$ corresponds to the norm of one of $U_1:=\{m_1 + i n_1 , -m_1 + i n_1 , m_1 - i n_1 , -m_1 - i n_1\}$, similarly for $m_2^2+n_2^2$, we have $U_2$. Hence, since they have the same norm, any element of $U_1$, say $z$, and $U_2$, say $w$, can be written as
\begin{align*}
z &= (a+ ib)(c+ i d), \\
w &= u (a+ib)(c-id),
\end{align*}
where $u\in \{1,-1 , i , -i\}$. Since we can work with $w u^{-1}$ instead of $w$, without loss of generality we can write
\begin{align*}
m_1+in_1 &= (a+ib)(c+id) , \\
m_2 + i n_2 &= (a+ ib)(c-id) .
\end{align*}
Consequently, the conditions on the sums become,
\begin{align*}
m_1 n_2 - m_2 n_1 &= -2(a^2+b^2)cd \\
m_1 m_2 + n_1 n_2 &= (c^2+d^2)(a^2-b^2)\\
N &= (a^2+b^2)(c^2+d^2)
\end{align*}
Since $N=(a^2+b^2)(c^2+d^2)$, thus their prime factors other than 2 are $1\mod 4$. So, we can use Lemma $\ref{r_2_inequality}$. \\
\noindent Hence, we can rewrite the inner sums in $\eqref{inner_sum_reference}$ as follows
$$ \sum_{\substack{(a^2+b^2)(c^2+d^2)\leq X \\ d_1 | cd \\ \gcd(d_1,N)=1 }} \frac{r_0^{k-2}(N)}{N} \text{ and } \sum_{\substack{(a^2+b^2)(c^2+d^2)\leq X \\ d_2 | (a^2-b^2) \\ \gcd(d_2,N)=1 }} \frac{r_0^{k-2}(N)}{N} . $$
We notice that we can turn the second sum to look like the first one by replacing $u=a-b , v=a+b$ and using $a^2+b^2=\frac{u^2+v^2}{2}$ and $a^2-b^2=uv$, hence, we only need to work on the first sum. Let $Y=\frac{X}{a^2+b^2}$. By Lemma $\ref{r_2_inequality}$ we have
\begin{align}
&\sum_{\substack{(a^2+b^2)(c^2+d^2)\leq X \\ d_1 | cd }} \frac{r_0^{k-2}(N)}{N} \leq \sum_{a^2+b^2 \leq X} \frac{r_0^{k-2}(a^2+b^2)}{a^2+b^2} \sum_{\substack{c^2+d^2\leq Y \\ d_1 | cd }} \frac{r_0^{k-2}(c^2+d^2)}{c^2+d^2} \nonumber \\
&\leq \sum_{a^2+b^2} \frac{r_0^{k-2}(a^2+b^2)}{a^2+b^2} \sum_{d_1 = d' d''} \bigg( \sum_{\substack{c^2+d^2\leq Y \\ d' | c }} \frac{r_0^{k-2}(c^2+d^2)}{c^2+d^2} \bigg)^{\frac{1}{2}} \bigg( \sum_{\substack{c^2+d^2\leq Y \\ d'' | d }} \frac{r_0^{k-2}(c^2+d^2)}{c^2+d^2} \bigg)^{\frac{1}{2}} \label{this_step},
\end{align}
by condition separation. We can bound the inner sums by Corollary $\ref{Henriot}$ and use $Y\leq X = x^{\frac{2}{3}}$ to get
\begin{align*}
\sum_{\substack{(a^2+b^2)(c^2+d^2)\leq X \\ d_1 | cd}} \frac{r_0^k(N)}{N}&\ll_k \sum_{a^2+b^2 \leq X} \frac{r_0^{k-2}(a^2+b^2)}{a^2+b^2} \sum_{d_1 = d' d''} \bigg( \frac{(\log X)^{2^{k-2}}}{d'^{\frac{1}{2}}} \bigg)^{\frac{1}{2}} \bigg( \frac{(\log X)^{2^{k-2}}}{d''^{\frac{1}{2}}} \bigg)^{\frac{1}{2}} \\
&\ll_k \sum_{a^2+b^2 \leq X} \frac{r_0^{k-2}(a^2+b^2)}{a^2+b^2} \sum_{d_1 = d' d''} \frac{(\log x)^{2^{k-2}}}{d_1^{\frac{1}{4}}} \\
&\ll \frac{\tau(d_1)}{d_1^{\frac{1}{4}}} (\log x)^{2^{k-2}} \sum_{m\leq X} \frac{r_0^{k-1}(m)}{m} \ll_k \frac{(\log x)^{2^{k-1}}}{d_1^{\frac{1}{8}}} ,
\end{align*}
where we have used $\tau(d_1)\ll_{\varepsilon} d_1^{\varepsilon}$ with some small $\varepsilon>0$, and Lemma \ref{r_2_bound} to get the last inequality.

\noindent Now, we can go back in our argument in the same manner to get
$$ M_{d_{\{1,2\}}} \ll_k \frac{\tau(d_{\{1,2\}})}{d_{\{1,2\}}^{\frac{1}{16}}} (\log x)^{2^{k-1}} \ll \frac{(\log x)^{2^{k-1}}}{d_{\{1,2\}}^{\frac{1}{32}}} . $$
Therefore, for $M_2$, acting analogously for each $M_{d_{\{i,j\}}}$,
\begin{align*}
M_2 &\ll_k \sum_{\ell_2} \frac{\mu^2(\ell_2) (4k)^{\omega(\ell_2)}}{\ell_2} \sum_{\prod_{1\leq i<j\leq k}d_{\{i,j\}} = \ell_2} \prod_{i<j}\bigg(  \frac{(\log x)^{2^{k-1}}}{d_{\{i,j\}}^{\frac{1}{32}}} \bigg)^{\frac{1}{K}} \\
&\ll_k \sum_{\ell_2} \frac{\mu^2(\ell_2) (4k)^{\omega(\ell_2)}}{\ell_2} \sum_{\prod_{i<j}d_{\{i,j\}} = \ell_2} \frac{(\log x)^{2^{k-1}}}{\ell_2^{\frac{1}{32K}}} \\
&=(\log x)^{2^{k-1}}\sum_{\ell_2} \frac{\mu^2(\ell_2) (4k)^{\omega(\ell_2)}}{\ell_2}  \frac{\tau_K(\ell_2)}{\ell_2^{\frac{1}{32K}}} \\
&\ll_k (\log x)^{2^{k-1}} \sum_{\ell_2} \frac{\mu^2(\ell_2) (4k)^{\omega(\ell_2)} \tau_K (\ell_2)}{\ell_2^{1+\frac{1}{32K}}} \\
&= (\log x)^{2^{k-1}} \prod_{p} \bigg( 1 + \frac{4k \binom{k}{2}}{p^{1+\frac{1}{32K}}} \bigg) \ll_k (\log x)^{2^{k-1}} ,
\end{align*}
where we use $\tau_K(p)=K=\binom{k}{2}$.
\section{The $e$-adic Sum}\label{dyadic_sum}\label{eight_section}
Recalling the $e$-adic sum from $\eqref{main_upper_bound}$ and using $\eqref{Singular_Series}$, we get
$$ \sum_{\log y \leq \ell \leq \log x^{1/3}} \frac{e^{\ell}}{\ell^{2k+1}} \sum_{\substack{(\vect{m}, \vect{n})\in \CA(\frac{x}{e^{\ell}}) \\ P^+(N)\leq e^{\ell+1}}} \mathfrak{S}(m_1 , \cdots , n_k) \ll_k \sum_{\log y \leq \ell \leq \log x^{1/3}} \frac{e^{\ell}}{\ell^{2k+1}} (M_1'+M_2') ,$$
where
\begin{align*}
M_1' = \sum_{\ell_1} \frac{\mu^2(\ell_1) (4k)^{\omega(\ell_1)}}{\ell_1} \sum_{\substack{(\vect{m} , \vect{n})\in \CA( \frac{x}{e^{\ell}}) \\ P^+(N)\leq e^{\ell+1} \\ \ell_1 |N}} 1 \\
M_2' = \sum_{\ell_2} \frac{\mu^2(\ell_2) (4k)^{\omega(\ell_2)}}{\ell_2} \sum_{\substack{(\vect{m}, \vect{n})\in \CA ( \frac{x}{e^{\ell}} ) \\ P^+(N)\leq e^{\ell+1} \\ \ell_2 | R(\vect{m}, \vect{n})}} 1
\end{align*}
similar to the ones we considered in Section $\ref{sixth_section}$ and $\ref{seventh_section}$. The proof of bound will proceed along the same lines, but now instead of logarithmic average, we work with the mean value. Our aim is to show that $M_i ' \ll_k \frac{x/e^{\ell}}{\log x}\ell^{2^{k-1}}$ for $i=1,2$. This bound would give us for the $e$-adic sum
$$ \sum_{\log y \leq \ell \leq \log x^{1/3}} \frac{e^{\ell}}{\ell^{2k+1}} \sum_{\substack{(\vect{m}, \vect{n})\in \CA(\frac{x}{e^{\ell}}) \\ P^+(N)\leq e^{\ell+1} }}  \mathfrak{S}(m_1 , \cdots , n_k)  \ll_k \frac{x}{\log x} \sum_{\log y \leq \ell \leq \log x^{1/3}} \ell^{2^{k-1}-2k-1}, $$
which is $x(\log x)^{2^{k-1}-2k-1}$ if $2^{k-1} -2k-1\geq 0$, i.e. for $k\geq 5$. However, it gives $\frac{x}{\log x}(\log y)^{2^{k-1}-2k}$ when $2^{k-1}-2k-1\leq -2$, i.e. for $k\leq 3$. For $k=4$, this gives
$$ \frac{x}{\log x} \sum_{\log y \leq \ell \leq \log x^{\frac{1}{3}}} \frac{1}{\ell} \ll x \frac{\log\log\log x}{\log x}. $$ 
Therefore, to finish the proof of Theorem $\ref{main_result}$, we just need to show that $M_i' \ll_k \frac{x/e^{\ell}}{\log x} \ell^{2^{k-1}}$ for $i=1,2$. \\
\noindent Acting similarly to $M_1$, see Section $\ref{sixth_section}$, for $M_1'$ we have
\begin{align*}
\sum_{\substack{(\vect{m}, \vect{n})\in \CA(\frac{x}{e^{\ell}}) \\ P^+(N)\leq e^{\ell+1} \\ \ell_1 |N}} 1 &\leq \sum_{\substack{N\leq \frac{x}{e^{\ell}} \\ P^+(N)\leq e^{\ell+1} \\ \ell_1 |N}} (4r_0(N))^k \\
&\ll_k r_0^k(\ell_1) \sum_{\substack{m\leq \frac{x}{\ell_1 e^{\ell}} \\ P^+(m)\leq e^{\ell+1}}} r_0^k(m) \\
&\ll_k \frac{r_0^k(\ell_1)}{\ell_1} \frac{x/e^{\ell}}{\log x/ e^{\ell}} (\ell+1)^{2^{k-1}} \\
&\ll \frac{r_0^k(\ell_1)}{\ell_1} \frac{x/e^{\ell}}{\log x} (\ell+1)^{2^{k-1}},
\end{align*}
where we used $\eqref{bound_for_representations}$ to get the first line, and Lemma \ref{r_2_inequality} for the second line with $N=m\ell_1$. We also used \eqref{specific_use} for the third line, and the fact that $e^{\ell}\leq x^{1/3}$ for the last line. Putting this in $M_1'$, we get
$$ M_1'\ll_k  \frac{x/e^{\ell}}{\log x}(\ell+1)^{2^{k-1}} \sum_{\ell_1} \frac{\mu^2(\ell_1) (4k)^{\ell_1} r_0^k(\ell_1)}{\ell_1^2}\leq \frac{x/e^{\ell}}{\log x} \ell^{2^{k-1}} \prod_p \bigg(1 + \frac{4k r_0^k(p)}{p^2} \bigg) \ll_k \frac{x/e^{\ell}}{\log x} \ell^{2^{k-1}}.  $$
Now, for the inner sum in $M_2'$, we proceed like we did for $M_2$. We have
$$ M_2' = \sum_{\ell_2} \frac{\mu^2(\ell_2) (4k)^{\omega(\ell_2)}}{\ell_2} \sum_{\substack{(\vect{m},\vect{n})\in \CA(\frac{x}{e^{\ell}}) \\ P^+(N)\leq e^{\ell+1} \\ \ell_2 | R(\vect{m},\vect{n})}} 1, $$
and use condition separation (keeping in mind that $R(\vect{m}, \vect{n})=\prod_{i<j}R(m_i,m_j,n_i,n_j)$, which is defined in $\eqref{Chaotic_Product}$) to get
$$ M_2' \leq \sum_{\ell_2} \frac{\mu^2(\ell_2) (4k)^{\omega(\ell_2)}}{\ell_2} \sum_{\prod_{i<j} d_{\{i,j\}} =\ell_2} \prod_{i<j} \bigg( \sum_{\substack{(\vect{m}, \vect{n})\in \CA(X) \\ P^+(N) \leq e^{\ell+1}\\ d_{\{i,j\}} | R(m_i , m_j , n_i , n_j) }} 1  \bigg)^{1/K} .  $$
We again expand in $\Z[i]$ to get sums over $(a^2+b^2), (c^2+d^2)$ and follow the same steps as in Section $\ref{seventh_section}$. In the sums $M_1'$ and $M_2'$ we work with the mean value instead of the logarithmic average. Therefore, we use Lemma $\ref{Henriot_use}$ instead of Corollary $\ref{Henriot}$ to bound our sums over $c^2+d^2\leq X$ in the step $\eqref{this_step}$. When we use Lemma $\ref{Henriot_use}$, we have
$$ \sum_{\substack{c^2+d^2 \leq \frac{x}{e^{\ell}} \\ d'|c \\ P^+(c^2+d^2)\leq e^{\ell+1} }} r_0^{k-2}(c^2+d^2)\ll_k \frac{x/e^{\ell}}{\log x/e^{\ell}} \frac{(\ell+1)^{2^{k-1}}}{d'^{\frac{1}{2}}} \ll_k \frac{x/e^{\ell}}{\log x} \frac{\ell^{2^{k-1}}}{d'^{\frac{1}{2}}} ,  $$
where we use $e^{\ell}\leq x^{1/3}$.
\noindent Let $K=\binom{k}{2}$ again, then,
\begin{align*}
\sum_{\substack{(\vect{m},\vect{n})\in \CA(\frac{x}{e^{\ell}}) \\ P^+(N)\leq e^{\ell+1} \\ \ell_2 |R(\vect{m}, \vect{n})}} 1 &\leq \sum_{\prod_{i<j} d_{\{i,j\}} =\ell_2} \prod_{i<j} \bigg( \sum_{\substack{(\vect{m}, \vect{n})\in \CA(X) \\ P^+(N) \leq e^{\ell+1}\\ d_{\{i,j\}} | R(m_i , m_j , n_i , n_j) }} 1  \bigg)^{1/K} \\
&\ll_k \sum_{\prod_{i<j} d_{\{i,j\}}=\ell_2} \prod_{i<j} \bigg( \frac{x/e^{\ell}}{\log x}\frac{\ell^{2^{k-1}}}{d_{\{i,j\}}^{1/32}} \bigg)^{1/K} \\
&= \frac{\tau_K(\ell_2)}{\ell_2^{\frac{1}{32K}}} \cdot \frac{x/e^{\ell}}{\log x} \ell^{2^{k-1}},
\end{align*}
analogously to $M_2$. Putting this in $M_2'$, we get
\begin{align*}
M_2' &\ll_k \frac{x/e^{\ell}}{\log x} \ell^{2^{k-1}} \sum_{\ell_2} \frac{\mu^2(\ell_2)(4k)^{\omega(\ell_2)} \tau_K(\ell_2)}{\ell_2^{1+\frac{1}{32K}}} \\
&=\frac{x/e^{\ell}}{\log x}\ell^{2^{k-1}} \prod_p \bigg(1+ \frac{4k\binom{k}{2}}{p^{1+\frac{1}{32K}}} \bigg) \ll_k \frac{x/e^{\ell}}{\log x}\ell^{2^{k-1}}.
\end{align*}
This concludes the proof of the upper bound in Theorem $\ref{main_result}$.
\section{Lower Bound}\label{ninth_section}
In this Section, we will prove the lower bound for $\CS_k$ in Theorem $\ref{main_result}$. Let $0<u\leq v<1$ and define
$$ \CB( x^u, x^v) := \{ (\vect{m}, \vect{n})\in \CA ( x^v ) : x^u \leq N \leq x^v , N \text{ squarefree} \}. $$
Then our lower bound in $\eqref{lower_bound_equation}$ is equal to
$$\CS_k \geq \sum_{(\vect{m}, \vect{n})\in \CB(x^{\frac{1}{3}} , x^{\frac{3}{8}})} \bigg( f_k \bigg( \frac{x}{N},  (\vect{m},\vect{n}) \bigg) - f_k(x^{\frac{1}{2}},  (\vect{m},\vect{n})) \bigg).$$
We remember the summands are $\geq 0$, so we further reduce the sum to the one over admissible representations $(\vect{m},\vect{n})$ and obtain
$$\CS_k \geq \sideset{}{^*}\sum_{(\vect{m}, \vect{n})\in \CB(x^{\frac{1}{3}} , x^{\frac{3}{8}})} \bigg( f_k \bigg( \frac{x}{N} ,  (\vect{m},\vect{n}) \bigg) - f_k(x^{\frac{1}{2}},  (\vect{m},\vect{n})) \bigg),$$
where $\sum\nolimits^*$ stands for a sum over admissible representations $(\vect{m},\vect{n})$. Now, we use Conjecture $\ref{conjectural_lower_bound}$ for $f_k(\frac{x}{N},  (\vect{m},\vect{n}))$ and Lemma $\ref{sieve_bound}$ for $f_k(x^{\frac{1}{2}},  (\vect{m},\vect{n}))$. This gives
\begin{align*}
&\gg_k  \sideset{}{^*}\sum_{(\vect{m}, \vect{n})\in \CB(x^{\frac{1}{3}} , x^{\frac{3}{8}})} \mathfrak{S}(m_1,\cdots , n_k) \frac{x}{N(\log x/N)^{2k+1}} + O_k \bigg( \mathfrak{S}(m_1,\cdots ,n_k) \frac{x^{\frac{1}{2}}}{(\log x)^{2k+1}} \bigg)  \\
&=  \sideset{}{^*}\sum_{(\vect{m}, \vect{n})\in \CB(x^{\frac{1}{3}} , x^{\frac{3}{8}})} \mathfrak{S}(m_1,\cdots , n_k) \frac{x}{N(\log x/N)^{2k+1}}\bigg( 1 + O_k \bigg( \frac{N(\log x/N)^{2k+1}}{x^{\frac{1}{2}}(\log x)^{2k+1}} \bigg) \bigg) \\
&\gg_k \frac{x}{(\log x)^{2k+1}} \sideset{}{^*}\sum_{(\vect{m}, \vect{n})\in \CB(x^{\frac{1}{3}} , x^{\frac{3}{8}})}  \frac{\mathfrak{S}(m_1,\cdots , n_k)}{N}\bigg( 1 + O_k \bigg( \frac{1}{x^{\frac{1}{8}}} \bigg) \bigg) \\
&\gg_k \frac{x}{(\log x)^{2k+1}}\sideset{}{^*}\sum_{(\vect{m}, \vect{n})\in \CB(x^{\frac{1}{3}} , x^{\frac{3}{8}})} \frac{\mathfrak{S}(m_1,\cdots , n_k)}{N} ,
\end{align*}
where we use $x^{1/3}\leq N \leq x^{3/8}$ for the third line. \\
\noindent Since $(\vect{m}, \vect{n})$ is an admissible representation, then by $\eqref{S_lower_bound}$, we have $\mathfrak{S}(m_1,\cdots ,n_k)\gg_k 1$. Hence
$$ \gg_k \frac{x}{(\log x)^{2k+1}}\sideset{}{^*}\sum_{(\vect{m}, \vect{n})\in \CB(x^{\frac{1}{3}} , x^{\frac{3}{8}})} \frac{1}{N} $$
Moreover, if $N$ is a squarefree integer with $2|N$ and $r_0(N)\geq C k \log^2 k$ for some $C>0$, then $N$ has $k$ many admissible representations by Lemma $\ref{admissible_lower_bound}$. Thus, we can lower bound the condition of an admissible representation to get
\begin{align*}
\sideset{}{^*}\sum_{(\vect{m}, \vect{n})\in \CB(x^{\frac{1}{3}} , x^{\frac{3}{8}})} \frac{1}{N} &\geq \sum_{\substack{(\vect{m}, \vect{n})\in \CB(x^{\frac{1}{3}} , x^{\frac{3}{8}}) \\ 2|N \\r_0(N)\geq C k \log^2 k}} \frac{1}{N} \\
&=  \sum_{\substack{(\vect{m}, \vect{n})\in \CB(x^{\frac{1}{3}} , x^{\frac{3}{8}}) \\2|N }} \frac{1}{N} - \sum_{\substack{(\vect{m}, \vect{n})\in \CB(x^{\frac{1}{3}} , x^{\frac{3}{8}})  \\ 2|N \\ r_0(N) < C k \log^2 k}} \frac{1}{N} \\
 &=\sum_{\substack{(\vect{m}, \vect{n})\in \CB(x^{\frac{1}{3}} , x^{\frac{3}{8}})  \\ 2|N }} \frac{1}{N} + O_k \bigg(  \sum_{\substack{x^{\frac{1}{3}}\leq N\leq x^{\frac{3}{8}}\\r_0(N) < C k \log^2 k}} \frac{r_0^k(N)}{N} \bigg) ,
\end{align*}
where for the last line, in the error term, we get rid of $2|N$ and $N$ squarefree conditions and use
$$ \sum_{(\vect{m} , \vect{n})\in \CA (N)} 1 \leq (4r_0(N))^k \ll_k r_0^k(N) ,  $$
like we did in $\eqref{bound_for_representations}$.

\noindent For the error term we can bound $r_0^k(N) \ll_k 1$ to get
$$ \frac{x}{(\log x)^{2k+1}} \bigg(\sum_{\substack{(\vect{m}, \vect{n})\in \CB(x^{\frac{1}{3}} , x^{\frac{3}{8}})   \\ 2|N \text{ and }N \text{ squarefree} }} \frac{1}{N} + O_k \bigg( \sum_{\substack{x^{\frac{1}{3}}\leq N\leq x^{\frac{3}{8}} \\ r_0(N) <C k \log^2 k}} \frac{1}{N} \bigg)\bigg) .$$
Then the sum in the error term is $\ll_k \log x$. For the main term, we can bound our number of representations by
$$ \sum_{\substack{(\vect{m}, \vect{n})\in \CA(N) \\ N=m_i^2+n_i^2 , \forall i\in [k]}} 1  = \binom{4r_0(N)}{k} \gg_k r_0^k(N) $$
by $\eqref{representations_to_N}$. This also adds a condition of $r_0(N)\geq k$, but we can get rid of it the same way we did with $r_0(N)\geq C k \log^2 k$ and get the same error term. Finally, we conclude
\begin{align*}
\CS_k &\gg_k \frac{x}{(\log x)^{2k+1}} \bigg( \sum_{\substack{x^{\frac{1}{3}}\leq N\leq x^{\frac{3}{8}}  \\ 2|N }} \mu^2(N) \frac{r_0^k(N)}{N} + O_k ( \log x ) \bigg) \\
&\gg_k x (\log x)^{2^{k-1}-2k-1},
\end{align*}
where we use Lemma \ref{lower_bound_r_2} to get the last line. This completes the proof of the lower bound and hence, Theorem \ref{main_result}.

\nocite{*}
\bibliographystyle{plain}
\bibliography{Bibliography}

\end{document}